\documentclass[nohypdvips]{siamart0516}
\usepackage[english]{babel}
\usepackage{pgf}
\usepackage{tikz}
\usepackage[utf8]{inputenc}
\usepackage{amsmath,amssymb,amsfonts}
\usepackage{enumerate}
\usepackage{blkarray}
\usepackage{cite}
\usepackage{mathtools}

\usetikzlibrary{calc,fit,matrix,arrows,automata,positioning,patterns}

\usepackage{hyperref}
\usepackage{xcolor}
\hypersetup{
    colorlinks,
    linkcolor={blue!50!black},
    citecolor={blue!50!black},
    urlcolor={blue!80!black}
}
\newcommand{\B}[1]{\mbox{\boldmath $#1$}}

\usepackage{pgfplots}
\usepackage{hyperref}

\newcommand\x{\times}

\newcommand\z{\colorbox{gray}{$\times$}}

\DeclareMathOperator{\rank}{rank}
\DeclareMathOperator{\rk}{rank}

\DeclarePairedDelimiter{\norm}{\lVert}{\rVert}

\newcommand{\tikzmark}[1]{\tikz[overlay,remember picture] \node (#1) {};}
\newcommand{\tikzdrawbox}[3][(0pt,0pt)]{%
    \tikz[overlay,remember picture]{
    \draw[#3]
      ($(left#2)+(-0.3em,0.9em) + #1$) rectangle
      ($(right#2)+(0.2em,-0.4em) - #1$);}
}

\newenvironment{code1}{%
  \mathcode`\:="603A  
  \def\colon{\mathchar"303A}
  \par
  \upshape
  \begin{list} 
    {} {\leftmargin = 0.0cm}
  \item[]
    \begin{tabbing}
      \hspace*{.3in} \= \hspace*{.3in} \=
      \hspace*{.3in} \= \hspace*{.3in} \=
      \hspace*{.3in} \= \hspace*{.3in} \= \kill
    }{\end{tabbing}\end{list}}

\newtheorem{remark}[theorem]{Remark}

\newcommand\TheTitle{Fast Hessenberg reduction of some rank structured matrices}
\newcommand\TheShortTitle{Hessenberg reduction of rank structured matrices}
\newcommand\TheAuthors{L. Gemignani, L. Robol}

\headers{\TheShortTitle}{\TheAuthors}

\title{\TheTitle\thanks{This work was partially supported by
    GNCS/IndAM, by the Research Council KU Leuven, project
    CREA/13/012, and by the Interuniversity Attraction Poles
    Programme, initiated by the Belgian State, Science Policy Office,
    Belgian Network DYSCO.}}

\author{L. Gemignani\thanks{Dipartimento di Informatica, Universit\`a
    di Pisa, Pisa, Italy, \email{l.gemignani@di.unipi.it}} \and
  L. Robol\thanks{Department of Computer Science, KU Leuven, Belgium,
    \email{leonardo.robol@cs.kuleuven.be}} }

\ifpdf
\hypersetup{
  pdftitle={\TheTitle},
  pdfauthor={\TheAuthors}
}
\fi

\begin{document}

\maketitle

\begin{abstract}
We develop two fast algorithms for Hessenberg reduction of a
structured matrix $A = D + UV^H$ where $D$ is a real or unitary
$n \times n$ diagonal matrix and $U, V \in\mathbb{C}^{n \times k}$.
The proposed algorithm for the real case exploits a two--stage
approach by first reducing the matrix to a generalized Hessenberg
form and then completing the reduction by annihilation of the
unwanted sub-diagonals.  It is shown that the novel method requires
$O(n^2k)$ arithmetic operations and it is significantly faster than
other reduction algorithms for rank structured matrices. The method
is then extended to the unitary plus low rank case by using a block
analogue of the CMV form of unitary matrices.  It is shown that a
block Lanczos-type procedure for the block tridiagonalization of
$\Re(D)$ induces a structured reduction on $A$ in a block staircase
CMV--type shape.  Then, we present a numerically stable method for
performing this reduction using unitary transformations and we show
how to generalize the sub-diagonal elimination to this shape, while
still being able to provide a condensed representation for the
reduced matrix.  In this way the complexity still remains linear in
$k$ and, moreover, the resulting algorithm can be adapted to deal
efficiently with block companion matrices.
\end{abstract}



\section{Introduction}
 
 Let $A = D + UV^H$ where $D$ is a real or unitary  $n \times n$ 
  diagonal matrix and 
  $U, V \in \mathbb{C}^{n \times k}$.  Such matrices do arise commonly in the numerical treatment of 
structured (generalized) eigenvalue problems \cite{ACL,AG}.  We consider the problem of reducing
  $A$ to upper Hessenberg form using unitary transformations, i.e., to find a unitary matrix
  $Q$ such that $QAQ^H = H = (h_{ij})$ and  $h_{ij} = 0$ for $i > j + 1$.   Specialized 
algorithms exploiting the rank structure of $A$ have complexity $O(n^2)$ whenever $k$ is a small 
constant independent of $n$ \cite{EGG_red,DB_red}.  However,  for applications to generalized eigenproblems 
the value of $k$ may be moderate or even large so that  it is worthwhile to ask for the algorithm 
to be cost efficient w.r.t. 
the size $k$ of the correction as well.

The Hessenberg 
reduction of a square matrix is the first basic step  in computing its eigenvalues.   Eigenvalue  computation 
for  (Hermitian) matrices modified by low rank perturbations is a classical problem  arising in many 
applications \cite{AG}.  More recently  methods for diagonal plus  low rank matrices have been used  
in combination with interpolation  techniques in order to solve 
generalized nonlinear eigenvalue problems \cite{ACL,BR}.  Standard Hessenberg  reduction algorithms for 
rank structured 
matrices \cite{EGG_red,DB_red} are both theoretically and practically 
ineffective as  the size of the correction  increases since their 
complexities depend quadratically  or even cubically on $k$, or
they suffer from possible instabilities \cite{bini2015quasiseparable}.
The aim of 
this paper is to describe a  novel efficient reduction scheme which attains 
the cost of $O(n^2k)$ arithmetic operations (ops).  For a real $D$ the incorporation of this scheme into the fast 
QR--based eigensolver  for rank--structured Hessenberg matrices given in \cite{EGG_alg} thus 
yields an eigenvalue method with  overall cost  $O(n^2k)$ ops. Further, in the unitary case  the design of such a scheme 
is an essential  step towards the development of an efficient eigenvalue solver for block companion matrices and, moreover,
the techniques employed in the reduction are  of independent interest  for computations with matrix valued orthogonal
polynomials on the unit circle \cite{arli,barel,simon}.

The proposed algorithm consists of two stages including a first
intermediate reduction to a generalized $\tau_k$-Hessenberg form\footnote{%
  We use the term \emph{generalized $\tau_k$-Hessenberg matrix} to mean
  a matrix with only $\tau_k$ non-zero subdiagonals. 
},
where $\tau_k=k$ in the real case and $\tau_k=2k$ in the unitary case.
A two-stage approach have been shown to be beneficial for blocked and
parallel computations \cite{bls,KKQQ}.  Our contribution is to show
that the same approach can conveniently be exploited in the framework
of rank--structured computations.  Specifically, to devise a fast
reduction algorithm we find $Q$ as the product of two unitary matrices
$Q = Q_2 Q_1$ determined as follows:
  \begin{description}
    \item[Reduction to banded form] 
      $Q_1 D Q_1^H$ is a banded matrix of bandwidth $\tau_k$ and
      $Q_1 U$ is upper triangular. In particular, this implies that
      the matrix $B := Q_1 A Q_1^H$ is in generalized $\tau_k$-Hessenberg form, 
      that is, its elements are $0$ below the $\tau_k$-th subdiagonal. 
    \item[Subdiagonal elimination] 
      We compute a unitary  matrix $Q_2$ such that 
      $Q_2 B Q_2^H$ is in  upper Hessenberg form.  The process employs a sequence of Givens rotations used 
       to annihilate the last  $\tau_k-1$ subdiagonals of the 
      matrix $B$. 
  \end{description}
  
  It is shown that both steps can be accomplished in $O(n^2k)$ ops.
  For a real $D$ the algorithm for the reduction to banded form relies
  upon an adaptation of the scheme proposed in \cite{AG} for bordered
  matrices.  The subdiagonal annihilation procedure employed at the
  second step easily takes advantage of the data--sparse
  representation of the matrices involved in the updating process.
  The extension to the unitary setting requires a more extensive analysis and
  needs some additional results.  By exploiting the relationships
  between the computation at stage 1 and the block Lanczos process we
  prove that the band reduction scheme applied to $\Re(D)$ and
  $U_D \colon = [ U , DU ]$ computes a unitary matrix $Q_1$ such that
  $F\colon =Q_1D Q_1^H$ is $2k$-banded in staircase form.  A suitable
  reblocking of $F$ in a block tridiagonal form puts in evidence
  analogies and generalizations of this form with the classical CMV
  format of unitary matrices \cite{BGE,killip2007cmv,cantero2003five},
  and this is why we refer to $F$ as a block CMV--type matrix.  An
  application of CMV matrices for polynomial rootfinding is given in
  \cite{BDCG}. The block CMV structure plays a fundamental role for
  reducing the overall cost of the reduction process.  To this aim we
  first introduce a numerically stable algorithm for computing $F$ in
  factored form using a suitable generalization of the classical Schur
  parametrization of scalar CMV matrices \cite{BGE}.  Then, we explain
  how bulge--chasing techniques can efficiently employed in the
  subsequent subdiagonal elimination stage applied to the cumulative
  $B$ to preserve the CMV structure of the part of the matrix that
  actually needs to be still reduced. Some experimental results are
  finally presented to illustrate the speed and the accuracy of our
  algorithms.

  The paper is organized as follows.  The real case is treated in
  Section~\ref{recomul}.  The two reduction steps applied to a real
  input matrix $D$ are discussed and analyzed in
  Subsection~\ref{sec:banded} and Subsection~\ref{sec:subdiagonal},
  respectively.  The generalization for unitary matrices is presented
  in Section~\ref{unicomul}.  In particular,
  Subsection~\ref{sec:unitary-properties} and \ref{sec:block-cmv}
  provide some general results on block CMV
  structures. Subsection~\ref{31},
  \ref{thm:block-diagonal-factorization} and \ref{32_1}
  deal with the construction of the block CMV structure in the
  reduction process whereas the computation and the condensed
  representation of the final Hessenberg matrix are the topics of
  Subsection~\ref{32} and Subsection~\ref{33}.  Experimental results
  are described in in Section~\ref{sec:numerical} followed by some
  conclusions and future work in Section~\ref{sec:conclusions}.

  \section{Reduction processes: The real plus low rank case}\label{recomul}
In this section we present a two phase   Hessenberg reduction algorithm  for a matrix 
$A = D + UV^H$ where $D$ is a real  $n \times n$ 
    diagonal matrix and  $U, V \in\mathbb{C}^{n \times k}$. 
    
\subsection{Reduction to banded form} \label{sec:banded}
      We show  how it is possible to transform the
      matrix $A$ into a Hermitian banded matrix plus a low rank 
      correction, i.e., $Q_1 A Q_1^H = D_1 + U_1 V_1^H$ 
      where
      \[
        D_1 = \begin{bmatrix}
          \times & \ldots & \times \\
          \vdots & \ddots &        & \ddots \\
          \times &        & \ddots &        & \times \\
                 & \ddots &        & \ddots & \vdots \\
                 &        & \times & \ldots & \times  \\
        \end{bmatrix}, \qquad 
        U_1 = \begin{bmatrix}
          X \\
          0
        \end{bmatrix}, \qquad X \in \mathbb{C}^{k \times k}. 
      \]
      and $D_1$ has bandwidth $k$.
The computation  can equivalently be reformulated as a band reduction problem   for 
the bordered matrix 
\[ 
E=\left[\begin{array}{c|c} 0_k & V^H\\\hline U & D\end{array}\right], 
\]
where $0_k$ denotes the zero matrix of size $k$.
 The algorithm we present  here is basically an adaptation  of 
Algorithm 2 in \cite{gragg_ammar} for solving this problem  efficiently. It  finds 
a unitary matrix $Q=I_k \oplus Q_1$ such that $Q E Q^H$ is in generalized $k$-Hessenberg form.

 The main idea behind the reduction
      is that  the matrix $D$ at the start is a (very special kind of)
      $k$-banded matrix, since it is diagonal. We then start to 
      put zeros in the matrix $U$ using Givens rotations and we
      right-multiply by the same rotations to obtain a unitary transformation. 
      This, generally, will degrade the banded structure of the matrix
      $D$. We then restore it by means of Givens rotations acting on the
      left in correspondence of the zeros of $U$. Since the rotations
      leave the zeros unchanged, this also preserves the
      partial structure of $U$. 
      
      More precisely, let $\ell(i,j)$ be the function defined as
      \[
        \begin{array}{crcl}
         \ell(i,j): & \mathbb{N}^{2} & \longrightarrow & \mathbb{N} \\
                    & (i,j)          & \longmapsto     & k (n+j-i) + j\\
        \end{array}
      \]
      
      \begin{lemma}
        The function $\ell$ is injective from the subset
        $\mathcal I \subseteq \mathbb{N}^2$ defined by
        $\mathcal I = [1, n] \times [1, k] \cap \mathbb{N}^2$.
      \end{lemma}
      
      \begin{proof}
        Assume that $\ell(i,j) = \ell(i',j')$. Then we have that
        \[
          0 = \ell(i,j) - \ell(i',j') = k( (j-j') - (i-i') ) + j - j'.
        \]
        This implies that $j \equiv j'$ modulo $k$, but since $j \in [1,k]$
        we can conclude that $j = j'$. Back substituting this in the above
        equation yields that $i = i'$, thus giving the thesis. 
      \end{proof}
      
      Since $\ell$ is injective and $\mathbb{N}$ is totally ordered
      the map $\ell$ induces an order on the set $\mathcal I$. This is
      exactly the order in which we will put the zeros in the matrix
      $U$.  Pictorially, this can be described by the following:
      \[ \scriptsize
        \begin{bmatrix}
			\times & \times & \times \\
			\times & \times & \times \\
			\times & \times & \times \\
			\times & \times & \times \\
			\times & \times & \times \\
			\times & \times & \times \\
			\times & \times & \times \\
			\times & \times & \times \\			
        \end{bmatrix} \rightarrow 
        \begin{bmatrix}
			\times & \times & \times \\
			\times & \times & \times \\
			\times & \times & \times \\
			\times & \times & \times \\
			\times & \times & \times \\
			\times & \times & \times \\
			\times & \times & \times \\
			0      & \times & \times \\			
        \end{bmatrix} \rightarrow
        \begin{bmatrix}
			\times & \times & \times \\
			\times & \times & \times \\
			\times & \times & \times \\
			\times & \times & \times \\
			\times & \times & \times \\
			\times & \times & \times \\
			0      & \times & \times \\
			0      & \times & \times \\			
        \end{bmatrix} \rightarrow
        \begin{bmatrix}
			\times & \times & \times \\
			\times & \times & \times \\
			\times & \times & \times \\
			\times & \times & \times \\
			\times & \times & \times \\
			\times & \times & \times \\
			0      & \times & \times \\
			0      & 0      & \times \\			
        \end{bmatrix} 
        \rightarrow
        \begin{bmatrix}
			\times & \times & \times \\
			\times & \times & \times \\
			\times & \times & \times \\
			\times & \times & \times \\
			\times & \times & \times \\
			0      & \times & \times \\
			0      & \times & \times \\
			0      & 0      & \times \\			
        \end{bmatrix} \rightarrow \dots
      \]
      The annihilation scheme proceeds by introducing the zeros from
      the bottom up along downwardly sloping diagonals.  Let
      $\mathcal R(A, (k,j))= G_{k} A G_k^H$, where
      $G_k=I_{k-2}\oplus \mathcal G_k \oplus I_{n-k}$ is a Givens
      rotation in the $(k-1,k)-$plane that annihilates $a_{k,j}$.  The
      band reduction algorithm is then described as follows:
  
\medskip
\noindent
\framebox{\parbox{8.0cm}{
\begin{code1}
{\bf Algorithm Band$\_$Red}\\
\ {\bf for} $i= 1 \ {\colon}\  n-1$\\
\ \ {\bf for} $j= 1\  {\colon} \ i$\\

\ \ \ $A=\mathcal R(A, \ell^{-1}(ki+j))$;  \quad \quad  \\
\ \ {\bf end}\\
\ {\bf end}
\end{code1}}}
\medskip

It is easily seen that we can perform each similarity transformation in 
$O(k)$ arithmetic  operations and, therefore, the overall  cost of the  band reduction is  
$O(n^2 k)$. 
      
\subsection{Subdiagonal elimination} \label{sec:subdiagonal} Our
algorithm for subdiagonal elimination and Hessenberg reduction
exploits the properties of bulge--chasing techniques applied to a band
matrix.  Observe that at the very beginning of this stage the input
matrix $A$ can be reblocked as a block tridiagonal matrix plus a
rank-$k$ correction located in the first block row.  At the first step
of the Hessenberg reduction we determine a unitary matrix
$\mathcal H_1\in \mathbb C^{k\times k}$ such that $\mathcal H_1 A(2:k+1, 1)$ is
a multiple of the of the first unit vector.  Then the transformed
matrix $A\colon =H_1 A H_1^H$, where
$H_1=1\oplus \mathcal H_1 \oplus I_{n-k-1}$, reveals a bulge outside
the current band which has to be chased off before proceeding with the
annihilation process.  In the next figure we illustrate the case where
$k=3$.
\[
  \begin{array}{llllllll}
    \x & \x & \x & \x & \x & \x & \x & \x   \\
  \tikzmark{left1}  \x   & \x & \x & \x & \x & \x  & \x & \x \\
     \x &  \x & \x & \x & \x & \x  & \x & \x \\
     \x \tikzmark{right1} &  \x &  \x &  \x & \x & \x  & \x & \x \\
     0 &  \x &  \x & \x & \x & \x  & \x & \x \\
     0 &  0 &  \x &  \x & \x & \x   & \x & \x \\
     0 &  0 &  0 & \x &  \x & \x  & \x & \x  \\
     0 &  0 &  0 &  0 &  \x &  \x   & \x & \x \\
  \end{array} \rightarrow 
 \begin{array}{llllllll}
    \x & \x & \x & \x & \x & \x & \x & \x   \\
  0  & \x & \x & \x & \x & \x  & \x & \x \\
     0 &  \x & \x & \x & \x & \x  & \x & \x \\
     0 &  \x &  \x &  \x & \x & \x  & \x & \x \\
     0 &   \tikzmark{left2}\z &  \z & \z & \x & \x  & \x & \x \\
     0 &  \z &  \z &  \z & \x & \x   & \x & \x \\
     0 &  \z &  \z & \z  \tikzmark{right2}&  \x & \x  & \x & \x  \\
     0 &  0 &  0 &  0 &  \x &  \x   & \x & \x \\
  \end{array}
\tikzdrawbox{1}{thick,red}
  \tikzdrawbox{2}{thick,green}
\]
The  crucial observation is that the bulge inherits the  rank structure of the matrix $\mathcal H_1^H$.  If we construct 
$\mathcal H_1$ 
as product of consecutive Givens rotations we find that $\mathcal H_1^H$ is  unitary lower Hessenberg and therefore  
the lower triangular part of the bulge has rank one at most.  Then the bulge can be moved down by computing its QR decomposition. 
From the rank property it follows that the $Q$ factor is again a unitary Hessenberg matrix so that the process can be 
continued until the bulge disappears. By swapping the elementary Givens  rotations employed in  both the annihilation and 
bulge--chasing step we arrive at the following algorithm. 
 
\medskip
\noindent
\framebox{\parbox{8.0cm}{
\begin{code1}
{\bf Algorithm Subd$\_$El}\\
\ {\bf for} $j= 1 \ {\colon}\  n-2$\\
\ \ {\bf for} $i= \min\{n, j + k\}\  \colon -1 \colon \ j + 2$ \quad \\

\ \ \ $A=\mathcal R(A, (i,j))$;   \\
\ \ \  {\bf for} $s= i+k\  \colon k \colon \ n$\\
\ \ \  \  $A=\mathcal R(A, (s,s-k-1))$;\\
\ \ \  {\bf end}\\
\ \ {\bf end}\\
\ {\bf end}
\end{code1}}}
\medskip

The arithmetic cost of the  algorithm is upper bounded by $\displaystyle\lceil \frac{n}{k}\rceil n k p $ where $p$ is the maximum number of 
arithmetic operation required to perform the updating $A\rightarrow A\colon = \mathcal R(A, (i,j))$  at each step  of 
the algorithm.  A data--sparse representation of the matrix $A$ is provided by the following. 

\begin{lemma} \label{lem:full-hess-recovery}
        At each step of the algorithm the matrix $A$ satisfies 
\[ 
A- A^H=UV^H-VU^H, 
\]
for suitable $U, V\in \mathbb C^{n\times k}$.
      \end{lemma}
      
      \begin{proof}
        The input matrix $A=D_1 +U_1 V_1^H$ satisfies 
$A- A^H=U_1V_1^H-V_1U_1^H$ and the relation is maintained under the congruence transformation 
$A\rightarrow A\colon = \mathcal R(A, (i,j))$.
 \end{proof}

Let $\Phi\colon \mathbb C^{n\times n}\times \mathbb C^{n\times k}\times \mathbb C^{n\times k} \rightarrow \mathbb C^{n\times n}$ be defined as 
follows:
\[
\Phi(A, U, V)={\mbox {\rm tril}}(A) +({\mbox {\rm tril}}(A,-1))^H +({\mbox {\rm triu}}(UV^H-VU^H,1)).
\]
Based on the previous result we know that the above  map provides
an efficient parametrization of the matrix $A$, so we 
can replace the (unstructured) transformation 
$A\rightarrow A\colon = \mathcal R(A, (i,j))$ in  Algorithm  {\bf Algorithm Subd$\_$El} 
with its condensed variant 
\[
\begin{array}{lll}
({\mbox {\rm tril}}(A), U, V)\rightarrow ({\mbox {\rm tril}}(A), U, V)\colon = \mathcal R_{\Phi}(\Phi( {\mbox {\rm tril}}(A), U, V)), (i,j)), \\ \\
 \Phi(\mathcal R_{\Phi}(\Phi( {\mbox {\rm tril}}(A), U, V)), (i,j)))=\mathcal R(\Phi( {\mbox {\rm tril}}(A), U, V)), (i,j))=
\mathcal R(A, (i,j)), 
\end{array}
\]
which can be employed by updating only the generators of the data--sparse representation   at the cost of $p=O(k)$ operations
 thus yielding the overall cost of $O(n^2 k)$ ops. 
 
\section{Reduction processes: The unitary plus low rank case}\label{unicomul}

In this section we deal with an extension of the
approach to the case of unitary plus low rank matrices. 
More precisely, we consider the case where $A = D + UV^H$
with $D$ being an   $n\times n$   unitary diagonal  matrix  and  $U, V \in\mathbb{C}^{n \times k}$.

A strong motivation for the interest in this class
of matrices comes from the companion block forms. 
If $A(x) = \sum_{i = 0}^n A_i x^i$ with $A_i \in \mathbb{C}^{m \times m}$ is a matrix polynomial then
the linear pencil
\[
  L(x) = x\begin{bmatrix}
    A_n \\
    & I_m \\
    & & \ddots \\
    & & & I_m \\
  \end{bmatrix} - 
  \begin{bmatrix}
    -A_{n-1} & \ldots & -A_1 & -A_0 \\
    I_m        & && \\
    & \ddots && \\
    && I_m &  \\
  \end{bmatrix}
\]
has the same spectral structure of $A(x)$
\cite{gohberg1982matrix}. When the polynomial is monic the same holds
for the above pencil (and so the generalized eigenvalue problem
associated with it turns into a standard eigenvalue problem) and the
constant term can be written as $Z + UW^H$ where
$U = e_1 \otimes I_m$, $W^H = [ -A_{n-1}, \ \ldots, \ -A_0 -I ]$ and
$Z = Z_n \otimes I_m$ where $Z_n$ is the $n \times n$ shift matrix
that maps $e_n$ to $e_1$ and $e_i$ to $e_{i+1}$ for $i < n$. In
particular, $Z_n$ can be easily diagonalized so that we can consider
an equivalent eigenvalue problem  for the matrix  $D + \tilde U \tilde W^H$ where  $D$  is  unitary  diagonal
generated from the $n-$th roots of unity.

\subsection{Properties of unitary plus low rank matrices}
\label{sec:unitary-properties}

The results stated in the previous section mainly relied upon the 
fact that the matrix $D$ is real in order to guarantee
that, at each step of the reduction,  the transformed matrix $PDP^H$ is
Hermitian, where $P$ denotes the product of all Givens rotations  accumulated so far.
In particular,  for a real $D$ maintaining the matrix
$PDP^H$ lower banded ensures that the whole matrix
is  also upper banded (thanks to the Hermitian
structure). This is not true anymore in the unitary 
case. In fact, applying the previous algorithm
``as it is'' leads to a matrix $PDP^H$ which is  in  
generalized $k$-Hessenberg form, that is, it is
lower banded (with bandwidth $k$)
but  generally dense in its upper  triangular part. 

It is possible to prove, relying on the formalism of
quasiseparable matrices, that a structure is indeed present
in the matrix $PDP^H$, even if it is not readily available to
be exploited in numerical computations.  Let us first recall some 
preliminaries (see \cite{eidelman:book1,eidelman:book2} for a survey of properties 
of quasiseparable matrices).

\begin{definition} \label{def:quasiseparable-matrix}
  Let $A  \in \mathbb{C}^{n \times n}$ a matrix. We say
  that $A$ is $(k_l, k_u)$-quasiseparable if, for any 
  $i = 1, \ldots, n - 1$, we have
  \[
    \rk(A[1:i,i+1:n]) \leq k_u, \qquad 
    \rk(A[i+1:n,1:i]) \leq k_l, 
  \]
  where we have used the MATLAB notation to 
  identify off-diagonal submatrices of $A$. 
\end{definition}

Banded matrices with lower bandwidth $k_l$ 
and upper bandwidth $k_u$
are a subset of $(k_l,k_u)$-quasiseparable matrix. 
This is a direct
consequence of the fact that, in a banded matrix, 
each off-diagonal submatrix has all the non-zero
entries in the top-right (or bottom-left) corner, and 
its rank is bounded by the bandwidth of the large
matrix. However, the class of quasiseparable matrices enjoys some
nice properties that the one of banded matrices does not have, 
such as closure under inversion. 

Before going into the details 
we introduce a very simple but important
result that is relevant to this setting, called
the Nullity theorem \cite{gustafson1984note,FM}. 

\begin{theorem}[Nullity] \label{thm:nullity} Suppose
  $A\in \mathbb C^{n\times n}$ is a non-singular matrix and let
  $\B \alpha$ and $\B \beta$ to be nonempty proper subsets of
  $\mathbb I_{n}\colon =\{1,\ldots, n\}$.  Then
\[
\rank(A^{-1}[\B \alpha; \B \beta])=\rank(A[\mathbb I_{n}\setminus \B \beta; \mathbb I_{n}\setminus \B \alpha])
 + |\B \alpha| +|\B \beta| -n, 
\]
where, as usual, $|J|$ denotes the cardinality of the set $J$. 
In particular,  if $A$ and $A^{-1}$ are  partitioned conformally as follows
  \[
    A = \begin{bmatrix}
      A_{1,1} & A_{1,2} \\
      A_{2,1} & A_{2,2} \\
    \end{bmatrix}, \qquad 
    A^{-1} = \begin{bmatrix}
      B_{1,1} & B_{1,2} \\
      B_{2,1} & B_{2,2} \\
    \end{bmatrix}
  \]
  with square diagonal blocks,  then the rank of each off-diagonal block of $A$
  matches the rank  of the corresponding block in the inverse. 
\end{theorem}

Hermitian and unitary quasiseparable matrices share an additional
property, which we call {\em rank symmetry}, which makes them
suitable to develop fast algorithms. 

\begin{theorem} \label{thm:unitary-matrices-rank-symmetry}
  Let $U$ be a unitary or Hermitian
  $n \times n$ matrix. Then 
  if $U$ is lower quasiseparable of rank $k$ it is also
  upper quasiseparable with the same rank. 
\end{theorem}

\begin{proof}
  If $U$ is Hermitian the result is obvious since each
  off-diagonal submatrix in the lower triangular part
  corresponds to another submatrix in the upper
  triangular part.   
  In the unitary case, instead, we can rely on the 
  Nullity Theorem  that guarantees
  that each submatrix of the form $P[i+1:n,1:i]$
  has rank equal to the same submatrix in the inverse of
  $P$. Since $P^{-1} = P^H$ we have the thesis.
\end{proof}

\begin{remark} \label{rem:rank-structure}
  If $H$ is a generalized $k$-Hessenberg matrix of the form
  $A = D + UV^H$ where $D$ is unitary and 
  $U, V \in \mathbb{C}^{n \times k}$ then
  \[
    H = PAP^H = PDP^H + (PU)(PV)^H. 
  \]
Suppose that $PDP^H$ is lower banded with bandwidth $k$ and  $PU$ is upper triangular. Since
  $PDP^H$ is rank symmetric we can conclude that
  the upper quasiseparable rank of $H$ is bounded
  by $\tau_k=2k$. 
\end{remark}

\subsection{The block CMV structure}
\label{sec:block-cmv}

The results of the previous section are not
constructive nor computationally helpful (at least in this form),
since they do not provide a representation for the upper part of the
matrix $PDP^H $ and a fortiori of the cumulative matrix $H$.  In the
spirit of \cite{bevilacqua2015compression} we pursue
a different approach by showing how to induce a particular banded structure  of width  $2k$
(the
so-called CMV shape \cite{killip2007cmv,BDCG}) on the unitary
matrix $PDP^H$. In this form the quasiseparable structure of $PDP^H$,
even if not of minimal size, will be evident by the banded structure
and, therefore, easily recoverable and representable.

\begin{definition}[CMV shape]\label{def:cmvshape}
  A unitary matrix $A$ is said to be \emph{CMV structured with block size $k$}
  if there exist $k \times k$  non-singular 
  matrices $R_i$ and $L_i$, respectively upper and lower triangular,
  such that 
  \[
    A = \begin{bmatrix}
          \times & \times & L_3  \\
          R_1 & \times & \times \\
            & \times & \times & \times & L_5\\
            & R_2 & \times & \times & \times \\
            &     &        & \times & \times \\
            &     &        & R_4    & \times & \ddots \\
            &     &        &        & \ddots & \ddots \\
        \end{bmatrix}
  \]
  where the symbol $\times$ has been used to identify 
  (possibly) nonzero blocks. 
\end{definition}

In order to simplify the notation we often assume that
$n$ is a multiple
of $2k$, so the above structures fit ``exactly'' in the matrix. However,
this is not restrictive  and the theory presented here continue
to hold in greater generality. In practice, one can deal with the
more general case by allows the blocks in the bottom-right corner
of the matrix to be smaller. 

\begin{remark}\label{cmvrem}
Notice that a matrix in CMV form with blocks of size $k$ is, 
in particular, $2k$-banded. The CMV structure with blocks
of size $1$ has been proposed as a generalization of what the
tridiagonal structure is for Hermitian matrices in 
\cite{cantero2003five} and \cite{killip2007cmv}.  A further analogy between the scalar and the block case is 
derived from the Nullity Theorem.  We have for $p>0$:
\[
0=\rk(A[1:2pk, 2(p+1)k+1: n])=\rk(A[2pk+1:n, 1:2(p+1)k])-2k
\]
which gives 
\[
\rk(A[2pk+1:2(p+1)k, (2p-1)k+1:2(p+1)k])=k.
\]
Pictorially we are setting  rank  constraints  on the following blocks
\[
 A = \begin{bmatrix}
          \times & \times & L_3  \\
          R_1 & \times & \times \\
            & \tikzmark{left1}\z & \z & \times & L_5\\
            & \z & \z  \tikzmark{right1}& \times & \times \\
            &     &        & \tikzmark{left2}\z & \z \\
            &     &        & \z    & \z  \tikzmark{right2} & \ddots \\
            &     &        &        & \ddots & \ddots \\
    \end{bmatrix}
    \tikzdrawbox{2}{thick,green}
    \tikzdrawbox{1}{thick,green}
  \]
and by  similar arguments on   the corresponding blocks in the  upper triangular portion. 
\end{remark}


\subsection{Existence of the CMV structure}\label{31}
This section is devoted to show that the reduction process
simultaneously applied to the diagonal matrix $\Re(D)$ and the
rectangular one $U_D= [ U , DU ]$ induces a CMV structure on the matrix
$A=D + UV^H$.  Observe that the construction of a unitary $P$ such
that $ T_D:=P\Re(D)P^H$ is $2k$-banded and, moreover, $P U_D$ is upper
triangular can theoretically be analyzed in the framework of block
Krylov (Lanczos) methods for converting $\Re(D)$ in block tridiagonal
form having fixed the first block column of $P^H$.  Exact breakdowns
occur when the iteration can not be continued due to a rank drop
and some modifications are required to generate a block
tridiagonal $T_D$ with subdiagonal blocks of varying (decreasing)
sizes. Conversely, we characterize the regular case where there is no
breakdown (and no rank reduction) in the corresponding block Lanczos
process by stating the following:
\begin{definition}[No breakdown condition]
  We say that the  matrix $A\in \mathbb C^{n\times n}$ satisfies 
  the \emph{no-breakdown-condition} related to $U\in \mathbb C^{n\times p}$ and $P\in \mathbb C^{n\times n}$ unitary  if
  $P U$ is upper triangular  of full rank and $H = P A P^H=(H_{i,j})$ is  a 
  block upper Hessenberg matrix,  
$H_{i,j}\in \mathbb C^{p\times p}$,  $1\leq i,j\leq s$, 
  and $H_{j+1,j}$  with  $1\leq j\leq s-1$  are of full rank.
\end{definition}

In order to simplify our  analysis of the reduction process
simultaneously applied to the diagonal matrix $\Re(D)$ and the
rectangular one $U_D= [ U , DU ]$  we assume 
the following two conditions are satisfied:
\begin{enumerate}
\item  The matrix $\Re(D)$ satisfies the 
  no-breakdown-condition related to $U_D$ and
  the unitary matrix  $Q^H\in \mathbb C^{n\times n}$ computed at end of the band  reduction process such that
  $Q^H \Re(D) Q =T_D$ is block tridiagonal.
\item  The dimension $n$  of $\Re(D)$ is a multiple of $2k$, i.e., $n=2k\ell$.  
This requirement ensures that all the  blocks of $T_D$ have 
constant size $p=2k$.
\end{enumerate}
We notice that the requirement $(1)$ is 
not so strict as it might seem. In fact, whenever a breakdown happens or, equivalently speaking, 
we find a rank loss in the associated Lanczos scheme,  then 
the reduction process  can be continued in such a way that 
the matrix obtained at the very end  will  still have a block  CMV structure with blocks of 
possibly varying sizes. However, this would  complicate the
handling of indices to track this situation and so we will 
not cover this case explicitly.

The next results  provides our algorithm for the reduction of a unitary diagonal $D$ into a 
block  CMV shape. 

\begin{lemma} \label{lem:cmv-structure}
  Let $B_R = Q^H \Re(D) Q$ be the $2k$-banded matrix obtained
  by applying the reduction scheme to the diagonal matrix
  $\Re(D)$ and the rectangular one $[ U \ DU ]$. Assume also
  that the no-breakdown-condition of $ \Re(D)$ related to
  $[ U \ DU ]$  and $Q^H$ is satisfied. Then
  the following structures are induced on $B_I$ and $B$
  defined as follows: 
  \begin{enumerate}[(i)]
    \item $B_I = Q^H \Im(D) Q$ is $2k$-banded. 
    \item $B = Q^H D Q$ is $2k$-banded and 
     CMV  structured with blocks of size $k$. 
  \end{enumerate}
\end{lemma}

\begin{proof}
  We can look at the $2k$-banded structure as a characterization
  of the operator $\Re(D)$ acting on $\mathbb C^{n}$. In
  particular, it tells us that, if we see the columns of $Q$
  as a basis for $\mathbb C^n$, the $j$-th one is mapped by
  $\Re(D)$ in the subspace spanned by the first $j + 2k$ ones. 
  Since, by construction, we have that 
  \[
    Q^H [ U \ DU ] = R 
    \iff [ U \ DU ] = Q \begin{bmatrix}
      \tilde R \\
      0_{2k} \\
      \vdots \\
      0_{2k} \\
    \end{bmatrix}, \quad \tilde R \ {\rm nonsingular},
  \]
  we have that the first $2k$ columns of $Q$ are a basis for the
  space generated by the columns of $[ U \ DU ]$. Under the
  no breakdown condition we have that the $(2k+1)$-th 
  column of $Q$ will be a multiple of the part of $\Re(D) U e_1$ 
  orthogonal to $[ U \ DU ]$. Notice that $\Re(D) = \frac{1}{2} (D + D^H)$ and therefore
  the $(2k+1)$-th 
  column of $Q$ will be  also a multiple of the part of $D^H  U e_1$ 
  orthogonal to $[ U \ DU ]$. Extending this remark yields that
  $Q$ is the unitary factor  in a $QR$ factorization of  a Krylov-type  matrix $\tilde U$, that is, 
  \begin{equation} \label{eq:spanning-spaces}
    Q R_U = \underbrace{\begin{bmatrix}
      U & DU & D^H U & D^2 U & {(D^2)}^H U & D^3 U & \ldots  {(D^{\ell-1})}^H U & D^\ell U
    \end{bmatrix}}_{\tilde U}, 
  \end{equation}
with $R_U$ invertible.
   Since $\Im(D) = \frac{1}{2i} (D - D^H)$ 
  we have that $Q^H \Im(D) Q$ has  a $2k$-banded
  structure and the $j$-th column of $Q$ is mapped  by  the linear operator $\Im(D)$ 
into  the span
  of the first $j+2k$ ones. Given that
  $D = \Re(D) + i \Im(D)$ we have that also $Q^H D Q$ is
  $2k$-banded. 
  
  It remains to show that $Q^H D Q$ is in fact endowed with
  a block CMV structure. For this it is convenient to rely upon  Equation~\ref{eq:spanning-spaces}. We can
  characterize   the $i$-th block column of $B$  by using 
  \[
    B E_i = Q^H D Q E_i  \  \Rightarrow  \  Q B E_i = D \tilde U R_U^{-1} E_i, 
  \]
where we set $E_i := e_i \otimes I_k$  and $D \tilde U$ can be explicitly written as 
  \[
    D \tilde U = \begin{bmatrix}
          DU & D^2U & U & D^3 U & D^H U & D^4 U & \ldots {(D^{\ell-2})}^H U & D^{\ell+1} U
        \end{bmatrix}.
  \]
  Recall that
  $R_U^{-1}$ is upper triangular, each block column of $\tilde U R_U^{-1}$
  is a combination of the corresponding one in $\tilde U$ and the previous ones. 
  Thus we  can deduce  that: 
  \begin{itemize}
    \item The block column of $\tilde U R_U^{-1}$ corresponding to the ones
    of the form $D^j U$ in $\tilde U$ (which are in position $2j$, except
    for the first one) are mapped into a 
    linear combination of the first $2(j+1)$  block columns by 
    the action of $D$ on the left. 
    \item Since the block column of the form $(D^H)^j U$ of $\tilde U$
    are mapped in $(D^H)^{j-1} U$ the corresponding  block column in $\tilde U R_U^{-1}$
    will be mapped in a combination of the previous ones. 
  \end{itemize}
  
  The above remarks give bounds on the maximum lower band structure that are stricter 
than the ones given by being $2k$-banded, and have the following shape:
  \[
    Q^H D Q = \begin{bmatrix}
      \times & \times & \times & \times & \times \\
      R_1 & \times & \times & \times & \times \\
        & \times & \times & \times & \times\\
        & R_2 & \times & \times & \times \\
        &     &        & \times & \times \\
        &     &        & R_4    & \times 
    \end{bmatrix}, 
  \]
where $R_i$   are suitable non-singular  upper triangular matrices.
  Performing a similar analysis on ${(Q^H D Q)}^H = Q^HD^HQ$ yields a similar
  structure also in the upper part, proving our claim about the CMV
  block structure of $Q^H D Q$. 
\end{proof}

The above description is theoretically satisfying, but applying
the Lanczos process to implicitly obtain the block CMV structure
is not ideal numerically.  In the next subsection  we introduce a different factored form
of block CMV structured matrices  which is the basis of  a numerically stable algorithm
for CMV reduction presented in Subsection \ref{32_1} . 

\subsection{Factorizing CMV matrices}

Block CMV matrices have interesting properties that we will exploit in
the following to perform the reduction. We show here that they can be
factorized in block diagonal matrices with $2k \times 2k$ diagonal blocks.

\begin{theorem}
  \label{thm:block-diagonal-factorization}
  Let $A\in \mathbb C^{n\times n}$, $n=2k\ell$,  be a block CMV matrix as in Definition~\ref{def:cmvshape}.
  Then there exist two block diagonal unitary matrices $A_1$ and $A_2$
  such that $A = A_1 A_2$, which have the following form:
  \[
    A_1 =
    \begin{bmatrix}
      \times & L_{A,2} \\
      R_{A,1}& \times    \\
      && \ddots \\
      &&           & \times & L_{A,2\ell} \\
      &&           & R_{A,2\ell-1}& \times    \\
    \end{bmatrix}, \quad 
    A_2 =
    \begin{bmatrix}
      I_k \\
      &\times & L_{A,3} \\
      &R_{A,2}& \times    \\
      &&& \ddots \\
      &&&           & \times \\
    \end{bmatrix}, 
  \]
  where the matrices $R_{A,j}$ are upper triangular and
  $L_{A,j}$ are lower triangular. 
  Moreover, for any $A_1$ and $A_2$ with the above structure, the
  matrix $A$ is block CMV. 
\end{theorem}

\begin{proof}
  The proof that we provide is constructive, and gives an algorithm
  for the computation of $A_1$ and $A_2$. Assume that $A$ has the following
  structure:
  \[
    A =
    \begin{bmatrix}
      A_{1,1} & A_{1,2} & L_3 &  \\
      R_1    & A_{2,2} & A_{2,3} &  \\
             & \vdots & \vdots & \ddots \\
    \end{bmatrix}, 
  \]
  and let $Q_1 \left[
    \begin{smallmatrix}
      S_1 \\ 0 
    \end{smallmatrix} \right] = \left[ \begin{smallmatrix}
      A_{1,1} \\ R_1 
    \end{smallmatrix} \right]$ be a QR decomposition
  of the first block column. Then we can factor $A$ as
  \[
    \begin{bmatrix}
      Q_1^H \\
      & I \\
    \end{bmatrix} A =
    \begin{bmatrix}
      S_1 & \tilde A_{1,2} & \tilde A_{1,3} \\
          & \tilde A_{2,2} & \tilde A_{2,3} \\
    \end{bmatrix}
  \]
  Since $S_1$ is upper triangular and $A$ is unitary we must have
  that $S_1$ is diagonal and $\tilde A_{1,2} = \tilde A_{1,3} = 0$
  because of the Nullity Theorem.
  Thus, we can assume to choose $Q_1$ so that $S_1 = I$ and we obtain
  \[
  A=  \begin{bmatrix}
      Q_1 \\
      & I \\
      & & \ddots \\
      & & & I \\
    \end{bmatrix}
    \begin{bmatrix}
      I \\
      & \tilde A_{2,2} & \tilde A_{2,3} \\
      & A_{3,2}        & A_{3,3} & A_{3,4} & R_5^H \\
      & R_2 & A_{4,3} & A_{4,4} & A_{4,5} \\
      & && \vdots & & \ddots 
    \end{bmatrix}.
  \]
  Notice that, if we look at $Q_1$ as a $2 \times 2$ block
  matrix the block in position $(2,1)$ has to be upper
  triangular, and we can force the one in position $(1,2)$
  to be lower triangular (the last $k$ columns are
  only determined up to right multiplication by a $k \times k$
  unitary matrix).

  We can continue the procedure by computing a QR factorization
  \[
    Q_2 \begin{bmatrix} S_2 \\ 0 \end{bmatrix} =
    \begin{bmatrix}
      A_{3,2} \\ R_2 \\
    \end{bmatrix} \implies
    Q_2^H
    \begin{bmatrix}
      A_{3,2} & A_{3,3} \\
      R_2    & A_{4,3} \\
    \end{bmatrix} = 
    \begin{bmatrix}
      S_2 & \times \\
      0 & 0 \\
    \end{bmatrix}
  \]
  since the right-handside has rank $k$  in view of Remark \ref{cmvrem} and we assume $R_2$ to
  be of full rank. This provides a new factorization
  \[
   A= \begin{bmatrix}
      Q_1 \\
      & Q_2 \\
      & & \ddots \\
      & & & I \\
    \end{bmatrix}
    \begin{bmatrix}
      I \\
      & \tilde A_{2,2} & \tilde A_{2,3} \\
      & S_2        & \times & \\
      &  &  & \tilde A_{4,4} & \tilde A_{4,5} \\
      & && \vdots & & \ddots 
    \end{bmatrix}, 
  \]
  where the upper zero structure follows from the lower
  one by re-applying the Nullity theorem again and from the fact that the
  triangular matrices are nonsingular. 
  Iterating this procedure until the end provides
  a factorization of the form: 
  \[    
    A = \begin{bmatrix}
      \times & L_{A,2} \\
      R_{A,1}& \times    \\
      && \ddots \\
      &&           & \times & L_{A,2\ell} \\
      &&           & R_{A,2\ell - 1}& \times    \\
    \end{bmatrix} \cdot 
    \begin{bmatrix}
      I_k \\
      &\times & \times \\
      &R_{A,2}& \times    \\
      &&& \ddots \\
      &&&           & \times \\
    \end{bmatrix}
  \]
  As a last step, we observe that the superdiagonal blocks in the
  right factor need to be lower triangular because of the presence
  of the lower triangular blocks in $A$, and for the full-rank
  condition that we have imposed on the triangular
  block entries. Therefore, the
  factorization has the required form.
  The other implication can be checked by a direct computation. 
\end{proof}

\subsection{From diagonal to CMV}\label{32_1}

We are now concerned with the transformation of a diagonal plus
low-rank matrix $A = D + UV^H$, with $D$ unitary, to a CMV plus
low-rank one where the block vector $U$ is upper triangular.  We have
proved in Section~\ref{31} that this transformation can be performed
under the no-breakdown condition.  Here we are going to present a
different numerically stable approach which works without restrictions
and outputs a factored representation of the block CMV matrix with
relaxed constraints on the invertibility of the triangular blocks.


In order to perform this reduction, we shall introduce some formal
concept that will greatly ease the handling of the structure.

\begin{definition}
  We say that a unitary matrix $\mathcal Q_j\in \mathbb C^{n\times n}$ is a {\em block unitary
    transformation} with block size $k$ if there exists an integer $s$
    and a $2k \times 2k$ unitary matrix $\hat Q_j$ such that
  \[
    \mathcal Q_j =
    \begin{bmatrix}
      I_{(j-1)k} \\ & \hat Q_j & \\ && I_{sk}
    \end{bmatrix}.
  \]
  The matrix $\hat Q_j$ is called the {\em active block} of
  $\mathcal Q_j$, and the
  integer $j=\frac n k -s-1$ is used to denote its position on the block diagonal. 
\end{definition}

Informally, we can see $\mathcal Q_j$ as a block version
of Givens rotations or, more generally, of essentially $2 \times 2$
unitary matrices. We restrict ourselves to the case where
$n$ is a multiple of $2k$, since that makes the definitions and
the notation much easier to follow. 

To illustrate our interest in block unitary transformations we give
the following simple application, which will be the base of our
reduction process.

\begin{lemma}
  \label{lem:block-qr-fact}
  Let $U\in\mathbb{C}^{n\times k}$, $n=\ell k$. Then there exists a sequence
  of block unitary transformations $\mathcal Q_1, \ldots, \mathcal Q_{\ell-1}$
  such that
  \[
    \mathcal Q_1 \ldots \mathcal Q_{\ell - 1} U = R, 
  \]
  with $R$ being a $n \times k$ upper triangular matrix. 
\end{lemma}

Block unitary transformations have some useful properties, that we
will exploit to make our algorithm asymptotically fast.

\begin{lemma}
  \label{lem:core-transf-properties}
  Let $\mathcal{A}_j, \mathcal{B}_j, \mathcal{C}_j \in \mathbb C^{n\times n}$ be block
  unitary transformations of block size $k$ and $D$ a unitary diagonal
  matrix of size $n=\ell k$ for some $\ell$. Then, 
  the following properties hold:
\begin{enumerate}[(i)]
\item If $|j - i| > 1$ then
  $\mathcal A_j \mathcal A_i = \mathcal A_i \mathcal A_j$.
\item There exists a block unitary transformation
  $\mathcal E_i$ such that $\mathcal E_i = \mathcal A_i \mathcal B_i$.
\item Given the sequence $\mathcal{A}_1, \ldots, \mathcal{A}_{\ell - 1}$,
  there
  exists a modified sequence of transformations $\tilde {\mathcal A}_1, \ldots
  \tilde {\mathcal A}_{\ell - 1}$ such that
  $\tilde {\mathcal A}_1 \ldots 
  \tilde {\mathcal A}_{\ell - 1} = \mathcal A_1 \ldots \mathcal A_{\ell - 1} D$.
\item For any choice of $\mathcal A_{j}, \mathcal B_{j + 1}, \mathcal C_{j}$
  with ``V-shaped'' indices, there exist three block unitary transformations
  $\tilde{\mathcal A}_{j+1}, \tilde{\mathcal B}_j, \tilde{\mathcal C}_{j+1}$
  such that $\mathcal A_j \mathcal B_{j+1} \mathcal C_{j} =
  \tilde{\mathcal A}_{j+1} \tilde{\mathcal B}_j \tilde{\mathcal C}_{j+1}$. 
\end{enumerate}
\end{lemma}

\begin{proof}
  All the properties can be checked easily using the definition
  of block unitary transformation. We only discuss (iv). Notice that
  the product of the block unitary transformations is a $3k \times 3k$
  unitary matrix embedded in an identity. Assume for simplicity that
  the size of the matrices is $3k$, so we do not have to
  keep track of the identities.
  
  Let $S := \mathcal A_1 \mathcal B_2
  \mathcal C_1$. We can compute its QR factorization by using
  $3$ block unitary transformations such that
  \[
    \tilde{\mathcal A}_2^H S =
    \begin{bmatrix}
      \times & \times & \times \\
      \times & \times & \times \\
      0_k    & \times & \times \\
    \end{bmatrix}, \quad 
    \tilde{\mathcal B}_1^H \tilde{\mathcal A}_2^H S =
    \begin{bmatrix}
      \times & \times & \times \\
      0_k & \times & \times \\
      0_k    & \times & \times \\
    \end{bmatrix}, \quad
    \tilde{\mathcal C}_2^H \tilde{\mathcal B}_1^H \tilde{\mathcal A}_2^H S =
    D, 
  \]
  where $D$ is a diagonal unitary matrix (since an upper triangular
  unitary matrix has to be diagonal). The diagonal term can be absorbed
  into the block unitary transformations, so we can assume without
  loss of generality that $D = I$. Thus, we have the new decomposition
  $\mathcal A_1 \mathcal B_2 \mathcal C_1 = S = \tilde {\mathcal A}_2
  \tilde {\mathcal B}_1 \tilde {\mathcal C}_2$, as desired. 
\end{proof}

A strict relationship exists between block CMV matrices
and block unitary transformations.

\begin{lemma}
  \label{lem:cmv-shape-in-core-transformations}
  A  unitary matrix $A \in \mathbb{C}^{n \times n}$, $n=\ell k$,   is CMV structured with block size $k$
 if and
  only if there exist block
  unitary transformations $\mathcal A_1, \ldots, \mathcal A_{\ell - 1}$
  such that 
  \[
    A = \begin{cases}
      \mathcal A_1 \mathcal A_3 \ldots \mathcal A_{\ell - 1}
      \mathcal A_2 \mathcal A_4 \ldots \mathcal A_{\ell - 2} &
      \text{ if } \ell \text{ is even} \\
      \mathcal A_1 \mathcal A_3 \ldots \mathcal A_{\ell - 2}
      \mathcal A_2 \mathcal A_4 \ldots \mathcal A_{\ell - 1} &
      \text{ if } \ell \text{ is odd}
    \end{cases}, 
  \]
  and all the block unitary transformations have
  active blocks of the form
  \[
    \begin{bmatrix}
      \times & L \\
      R      & \times \\
    \end{bmatrix}, \qquad
    R, L^H \text {nonsingular upper triangular}. 
  \]
\end{lemma}

\begin{proof}
  Assume that $\ell$ is even. The other case is handled in the same
  way. Let $A_1 := \mathcal A_1 \mathcal A_3 \ldots \mathcal A_{\ell - 1}$
  and $A_2 :=       \mathcal A_2 \mathcal A_4 \ldots \mathcal A_{\ell - 2}$.
  These two matrices are block diagonal and have the structure
  prescribed by Theorem~\ref{thm:block-diagonal-factorization}, so
  $A_1 A_2$ is block CMV.

  On the other hand, every CMV matrix
  can be factored as $A = A_1 A_2$, and the diagonal blocks
  of these two matrices have the structure required to be the active
  part of a block unitary transformation. Therefore,
  both $A_1$ and $A_2$ can be written as the product of
  a sequence of odd and even-indexed block unitary transformations,
  respectively. This concludes the proof. 
\end{proof}

\begin{remark}\label{cmvminus}
  If we remove the assumption about the invertibility of  the upper triangular blocks
  $R, L^H$  of the   unitary transformations $\mathcal A_j$ in the previous  lemma then  the factored
  representation  still  implies  a block CMV shape of the cumulative matrix $A$.  Indeed, these block CMV shaped 
  matrices are the ones  considered  in the  actual  reduction process (compare also with Lemma \ref{ccc}).
  \end{remark}

\begin{lemma}
  \label{lem:Vtodescending}
  Let $\mathcal A_1 , \ldots , \mathcal A_{\ell - 1} \in \mathbb C^{n\times n}$ and
  $\mathcal B_{\ell - 1}, \ldots, \mathcal B_1 \in \mathbb C^{n\times n}$, $n=\ell k$, be  two sequences of block
  unitary transformations of  block size $k$.
  Then, there exist an $n\times n$
  unitary matrix $P$ and a sequence of block unitary transformations  of  block size $k$
  $\mathcal C_1,\ldots, C_{\ell - 1}\in \mathbb C^{n\times n}$ such that
  \[
    P \mathcal A_1 \ldots \mathcal A_{\ell - 1} \mathcal B_{\ell - 1}
    \ldots \mathcal B_1 P^H = \mathcal C_1 \ldots \mathcal C_{\ell - 1}.
  \]
  Moreover, $P (e_1 \otimes I_k) = e_1 \otimes I_k$. 
\end{lemma}

\begin{proof}
  We prove the result by induction on $\ell$. 
  If $\ell = 1$ there is nothing to prove, and
  if $\ell = 2$ we can choose $P = I$ and we have
  \[
    \mathcal A_1 \mathcal B_1 = \mathcal C_1
  \]
  which can be satisfied thanks to property (ii) of
  Lemma~\ref{lem:core-transf-properties}. Assume now that the
  result is valid for $\ell - 1$, and we want to prove it for
  $\ell$. We can write
  \[
    A := \mathcal A_1 \ldots \mathcal A_{\ell - 1}
    \mathcal B_{\ell - 1} \ldots \mathcal B_1 =
    \mathcal A_1 A_2 \mathcal B_1, 
  \]
  where we have set
  $A_2 := \mathcal A_{2} \ldots \mathcal A_{\ell - 1} \mathcal B_{\ell
    - 1} \ldots \mathcal B_2$. Since the first block column and block
  row of $A_2$ is equal to the identity, we can rewrite it as
  $A_2 = I_k \oplus \tilde A$, with
  $\tilde A = \tilde{\mathcal A}_{1} \ldots \tilde{\mathcal A}_{\ell -
    2} \tilde{\mathcal B}_{\ell-2} \ldots \tilde{\mathcal B}_1$.
  By the inductive hypothesis we have that there exists a unitary
  matrix $\tilde P$ so that $\tilde P \tilde A \tilde P^H =
  \tilde{\mathcal C}_1 \ldots \tilde{\mathcal C}_{\ell - 2}$.

  If we set $P_2 := I_k \oplus \tilde P$ then $P_2 \mathcal A_1 =
  \mathcal A_1$ and $P_2 \mathcal B_1^H = \mathcal B_1^H$, so we have
  \begin{equation} \label{eq:chasing-V-shape-first}
    P_2 A P_2^H = \mathcal A_1 \mathcal D_2 \ldots \mathcal D_{\ell - 1}
    \mathcal B_1, 
  \end{equation}
  where we have set $\mathcal D_{j+1} := I_k \oplus \tilde{\mathcal C}_j$.
  We can move $\mathcal B_1$ to the left since it commutes
  with all the matrices except the first two, and by setting
  $\mathcal A_1 \mathcal D_2 \mathcal B_1 =
  \hat{\mathcal B}_2 \mathcal C_1 \hat{\mathcal D}_2$,
  thanks to property (iv) of Lemma~\ref{lem:core-transf-properties},
  we get
  \[
    P_2 A P_2^H = \hat{\mathcal B}_2 \mathcal C_1 \hat{\mathcal D}_2
    \mathcal D_3 \ldots \mathcal D_{\ell - 1}
  \]
  Left-multiplying by $\hat{\mathcal B}_2^H$ and right-multiplying
  by $\hat{\mathcal B}_2$ yields:
  \[
    ( \hat{\mathcal B}_2 P_2 ) A ( \hat{\mathcal B}_2 P_2 )^H =
    \mathcal C_1 \hat{\mathcal D}_2 \mathcal D_3 \ldots \mathcal D_{\ell - 1}
    \hat{\mathcal B}_2. 
  \]
  The part of the right-handside that follows $\mathcal C_1$, that is
  the matrix
  $\hat {\mathcal D}_2 \hat {\mathcal{D}}_3 \ldots \mathcal D_{\ell-1}
  \hat {\mathcal B}_2$ has the same form of the right-handside in
  Equation~(\ref{eq:chasing-V-shape-first}), but with one term less.
  We can reuse the
  same idea $\ell -
  3$ times and obtain the desired decomposition.
\end{proof}


We show here the special form required by the block unitary
transformation is not very restrictive. In particular, every time
that we have a CMV-like factorization for a matrix in terms
of block unitary transformations, we can always perform a unitary
transformation to obtain the required triangular structure
inside the blocks. 

\begin{lemma}\label{ccc}
  \label{lem:triangularization}
  Let $A\in \mathbb C^{n\times n}$, $n=\ell k$,  a unitary matrix that can be factored as
  \[
      A = \begin{cases}
      \mathcal A_1 \mathcal A_3 \ldots \mathcal A_{\ell - 1}
      \mathcal A_2 \mathcal A_4 \ldots \mathcal A_{\ell - 2} &
      \text{ if } \ell \text{ is even} \\
      \mathcal A_1 \mathcal A_3 \ldots \mathcal A_{\ell - 2}
      \mathcal A_2 \mathcal A_4 \ldots \mathcal A_{\ell - 1} &
      \text{ if } \ell \text{ is odd}
    \end{cases}, 
  \]
  with $\mathcal A_j\in \mathbb C^{n\times n}$ block unitary transformation of  block size $k$.
  Then, there exist a unitary transformation $P$, which is the direct
  sum of $\ell$ unitary blocks of size $k \times k$, such that $PAP^H =\tilde A$
  satisfies
  \[
      \tilde A = \begin{cases}
    \tilde  {\mathcal A}_1 \tilde{ \mathcal A}_3 \ldots \tilde {\mathcal A}_{\ell - 1}
      \tilde {\mathcal A}_2 \tilde {\mathcal A}_4 \ldots \tilde {\mathcal A}_{\ell - 2} &
      \text{ if } \ell \text{ is even} \\
      \tilde {\mathcal A}_1 \tilde {\mathcal A}_3 \ldots \tilde {\mathcal A}_{\ell - 2}
      \tilde {\mathcal A}_2 \tilde {\mathcal A}_4 \ldots \tilde {\mathcal A}_{\ell - 1} &
      \text{ if } \ell \text{ is odd}
    \end{cases}, 
      \]
      where the active blocks of $\tilde{\mathcal A}_j$
      are of the form
      \[
    \begin{bmatrix}
      \times & L \\
      R      & \times \\
    \end{bmatrix}, \qquad
    R, L^H \text { upper triangular}. 
  \]
\end{lemma}

\begin{proof}
  Assume that we are in the case $\ell$ even, and let
  us denote $A = A_{l} A_r =
  ( \mathcal A_1 \mathcal A_3 \ldots \mathcal A_{\ell - 1} ) \cdot 
  ( \mathcal A_2 \mathcal A_4 \ldots \mathcal A_{\ell - 2} )$. 
  It is not restrictive to
  assume that the blocks in $A_r$ are
  already in the required form. In fact, if this is not true, we
  can compute a block diagonal unitary matrix $Q$ (with $k \times k$
  blocks) such that $Q A_r$
  has the required shape. Then, by replacing
  $A_l$ with
  $A_l Q^H$ we get another
  factorization of $A$ where the right factor is already
  in the correct form.

  We now show that we can take the left factor in the
  same shape without deteriorating the structure of the right one.
  Let $Q_1$ be a unitary transformation operating on the first block row such
  that $Q_1 A_l$ has the block in position $(1,2)$ in lower triangular
  form. Then we have
  \[
    Q_1 A_l A_r Q_1^H =     Q_1 A_l Q_1^H A_r =  A_l^{(1)} A_r 
  \]
  since $A_r$ and $Q_1$ commute. Moreover, $A_l^{(1)}$ has the
  first block row with the correct structure.

  We now compute another
  unitary transformation $Q_2$ operating on the second block row
  such that $Q_2 A_l^{(1)}$ has the second row with the correct structure.
  Now the matrix $A_r Q_2^H$ loses the triangular structure in the block
  in position $(3,2)$. However, we can compute another transformation $P_3$
  operating on the third block row that restores the structure in
  $P_3 A_r Q_2^H$, and therefore we have
  \[
    Q_2 A_l^{(1)} A_r Q_2^H = (Q_2 A_l^{(1)} P_3^H) \cdot (P_3 A_r Q_2^H) 
    = A_l^{(2)} A_r^{(2)}
  \]
  since the right multiplication by $P_3^H$ does not degrade the
  partial structure that we have in $A_l^{(1)}$. We can iterate
  this
  process until all the blocks have the required structure,
  and this proves the lemma. 
\end{proof}

In the sequel we refer to the matrix $\tilde A$  defined in the  previous lemma as
a CMV structured unitary matrix with block size $k$ even if  there is no assumption about the
invertibility of the triangular blocks. We now have all the tools required to perform the initial reduction
of $U$ to upper triangular form and of $D$ to block CMV structure. 

\begin{theorem}
  \label{thm:diagonal-to-block-cmv}
  Let $D\in \mathbb C^{n\times n}$  be  a unitary diagonal 
  matrix and $U \in \mathbb C^{n\times k}$  with $n=\ell k$ for some $\ell \in \mathbb N$. Then,
  there exists a unitary matrix $P$ such that
  $PDP^H$ is CMV  structured with block size $k$  and
  $PU(e_1 \otimes I_k)=(e_1 \otimes I_k) U_1$. 
\end{theorem}

\begin{proof}
  The proof is divided in two stages. First, we show that we can
  build a unitary matrix $P_U$ such that $P_U U = R$ is upper
  triangular. Then, we use Lemma~\ref{lem:Vtodescending}
  to construct another unitary matrix $P_C=I_k \oplus \tilde P_C$ so that
  $P_C \cdot (P_U D P_U^H) \cdot P_C^H$ is in block CMV shape. We then set
  $P := P_C P_U$ and conclude the proof.

  In view of Lemma~\ref{lem:block-qr-fact} we can set
  $P_U = \mathcal A_{1} \ldots \mathcal A_{\ell - 1}$ so that $P_U U = R$.
  Applying the same transformation to $D$ yields
  \[
  P_U D P_U^H = \mathcal A_{1} \ldots \mathcal A_{\ell - 1} D \mathcal A_{\ell - 1}^H \ldots
  \mathcal A_{1}^H = \mathcal A_{1} \ldots \mathcal A_{\ell - 1} \mathcal B_{\ell - 1} \ldots \mathcal B_{1}, 
  \]
  where we have used property (ii) of Lemma~\ref{lem:core-transf-properties}
  to merge $D$ into the right block unitary transformations. We may now
  use Lemma~\ref{lem:Vtodescending} to obtain
  unitary transformations $\mathcal C_1, \ldots, \mathcal C_{\ell - 1}$
  and a unitary matrix $P_V$ 
  so that
  \[
    P_{V} \cdot (P_U D P_U^H) \cdot P_{V}^H =
    \mathcal C_1 \ldots \mathcal C_{\ell - 1}
  \]
  We now want to manipulate the above factorization to obtain one in
  CMV form, according to
  Lemma~\ref{lem:cmv-shape-in-core-transformations}. Let us
  left multiply the above by $\mathcal C_2 \ldots \mathcal C_{\ell-1}$
  and right multiply it by its inverse. We get:
  \[
    ( \mathcal C_2 \ldots \mathcal C_{\ell - 1} )
    \mathcal C_1 \ldots \mathcal C_{\ell - 1}
    ( \mathcal C_2 \ldots \mathcal C_{\ell - 1} )^H =
    \mathcal C_2 \ldots \mathcal C_{\ell - 1} \mathcal C_1. 
  \]
  Since $\mathcal C_1$ commutes with all the factors except $\mathcal C_2$,
  the above expression is equal to
  \[
    \mathcal C_2 \mathcal C_1 \mathcal C_3 \ldots \mathcal C_{\ell - 1}. 
  \]
  We can repeat the operation and move
  $\mathcal C_4 \ldots \mathcal C_{\ell - 1}$ to the left and, using the
  commutation properties, obtain
  \[
  (\mathcal C_4 \ldots \mathcal C_{\ell - 1})     \mathcal C_2 \mathcal C_1 \mathcal C_3
  \ldots \mathcal C_{\ell - 1} (\mathcal C_4 \ldots \mathcal C_{\ell - 1})^H =
  \mathcal C_4 \ldots \mathcal C_{\ell - 1} \mathcal C_2 \mathcal C_1 \mathcal C_3 = 
    \mathcal C_2 \mathcal C_4 \mathcal C_1 \mathcal C_3 \mathcal C_5 \ldots
    \mathcal C_{\ell - 1}. 
  \]
  We repeat this process until we obtain all the $\mathcal C_i$ with
  even indices on the left and the ones with odd indices
  on the right. Since this is the structure required by Lemma~\ref{lem:cmv-shape-in-core-transformations}
  and we have performed operations that
  do not involve $\mathcal C_1$, we can write this final step
  as a matrix $P_S$ with $P_S (e_1 \otimes I_k) = e_1 \otimes I_k$,
  so that
  \[
    P_S P_V P_U A ( P_S P_V P_U )^H = \mathcal C_2 \ldots \mathcal C_{\ell - 2}
    \mathcal C_1 \ldots \mathcal C_{\ell - 1}, 
  \]
  assuming $\ell$ is even. Since $P_S P_V (e_1 \otimes I_k) = e_1 \otimes I_k$
  by construction we have  that $P_S P_V P_U U = R$. We can
  compute another unitary block diagonal  transformation $P_T$ obtained through
  Lemma~\ref{lem:triangularization} to ensure the correct triangular 
  structure in the unitary blocks, and so the proof is
  complete by settings $P := P_T P_S P_V P_U$. 
  \end{proof}

We now comment on the computational cost of performing
the algorithm in the proof of Theorem~\ref{thm:diagonal-to-block-cmv}.
Assume that the matrices are of size $n$, the blocks are of size $k$,
and that $n$ is a multiple of $k$. 
The following steps are required:

\begin{enumerate}[(i)]
\item A block QR reduction of the matrix $U$. This requires
  the computation of $\mathcal O(\frac n k)$ QR factorizations
  of $2k \times 2k$ matrices. The total cost is
  $\mathcal O(\frac n k \cdot k^3) \sim \mathcal O(n k^2)$.
\item The application of Lemma~\ref{lem:Vtodescending}, which requires
  about $\mathcal O( (\frac n k)^2 )$ QR factorizations and
  matrix multiplications, all involving $2k \times 2k$ or
  $3k \times 3k$ matrices. Therefore, the total cost
  is of $\mathcal O(n^2 k)$ flops.
\item The manipulation of the order of the transformations
  requires again $\mathcal O( (\frac n k)^2 )$ operations
  on $2k \times 2k$ matrices, thus adding another
  $\mathcal O(n^2 k)$ flop count.
\item The final application of Lemma~\ref{lem:triangularization} requires
  the computation of $\mathcal O(\frac n k)$ QR factorizations,
  and thus another $\mathcal O(n k^2)$ contribution. 
\end{enumerate}

The total cost of the above steps is thus of $\mathcal O(n^2 k)$ flops, since
$k \leq n$, and  the dominant part is
given by steps (ii) and (iii).

\subsection{Preserving the CMV structure}\label{32}

In this subsection we deal with efficiently implementing the
final part of the algorithm: the reduction from block CMV form
to upper Hessenberg form.  In order to describe this part of
the algorithm we consider the factored form that we have presented
in the previous section, that is the factorization
(here reported for $\ell$ even)
\[
  A_{CMV} = \mathcal A_1 \ldots \mathcal A_{\ell-1}
  \mathcal B_{2} \ldots \mathcal B_{\ell - 2}.
\]
In the following we are interested in showing how,
using an appropriate bulge-chasing strategy, we can maintain
the sparse structure of $A_{CMV}$ during the Hessenberg reduction
procedure of $A_{CMV} + UV^H$.

Given that the block structure and the indexing inside the
blocks can be difficult to visualize, we complement the analysis
with a pictorial description of the steps, in which
the matrix has the following structure:
\[
  \begin{tikzpicture}[scale=.5]
    \node at (-2,-4) {$A_{CMV} =$};
     \draw (0,0) -- (2,0) -- (3,-1) -- (3,-2) -- (4,-2)
        -- (5,-3) -- (5,-4) -- (6,-4) -- (7,-5) -- (7,-7)
        -- (5,-7) -- (5,-6) -- (4,-6) -- (3,-5) -- (3,-4) -- (2,-4)
        -- (1,-3) -- (1,-2) -- (1,-2) -- (0,-1) -- (0,0); 
   \end{tikzpicture}
\]
Let $P_1$ be a sequence of $2k - 2$ Givens rotations that operate on
the rows with indices $2, \ldots, 2k$, starting from the bottom to the
top. These are the rotations needed to put the first column of
$A_{CMV} + UV^H$ in upper Hessenberg-form. Since we do not make any
assumption on the structure of $U$ and $V$ beyond having $U$ in upper
triangular form (the block $U_1$ in Theorem
\ref{thm:diagonal-to-block-cmv} can be absorbed in the factor $V$),
we do not restrict the analysis to the case where these rotations
create some zeros in $A_{CMV}$.

Our task is to determine a set of unitary transformations $P_2$, which
only operate on rows and columns with indices larger than $3$ (in
order to not destroy the upper Hessenberg structure created by $P_1$),
such that $P_2 P_1 A_{CMV} P_1^H P_2^H$ has the trailing
submatrix obtained removing the first row and column still
in block CMV shape.

We start by applying the first $k -1$ rotations in $P_1$ from both sides.
Applying them from the left degrades the upper triangular structure
of the block in position $(2,1)$, and makes it upper
Hessenberg. Pictorially, we obtain the following:
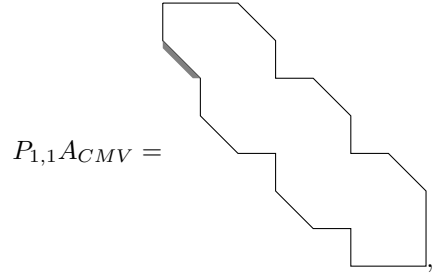
\[
  \begin{tikzpicture}[scale=.5]
    \node at (-2,-4) {$P_{1,1} A_{CMV} =$};
     \draw (0,0) -- (2,0) -- (3,-1) -- (3,-2) -- (4,-2)
        -- (5,-3) -- (5,-4) -- (6,-4) -- (7,-5) -- (7,-7)
        -- (5,-7) -- (5,-6) -- (4,-6) -- (3,-5) -- (3,-4) -- (2,-4)
        -- (1,-3) -- (1,-2) -- (1,-2) -- (0,-1) -- (0,0);
        \fill[gray] (0,-1.2) -- (.8,-2.0) -- (1,-2) -- (0,-1);
      \end{tikzpicture}, 
\]
where $P_{1,1}$ has been used to denote the first $k - 1$ rotations
in $P_1$, and the gray area mark the fill-in introduced in the matrix.
Applying them from the right, instead, creates some bulges in the upper
triangular blocks in position $(4,2)$.  These bulges correspond with a  rank one structure
in the lower triangular part  of this block
and can be  chased away by a sequence of $k-1$ rotations. We can chase them by multiplying
by these  rotations  from the left and from the right and  then proceed in the same way
until they
reach the bottom of the matrix as follows:
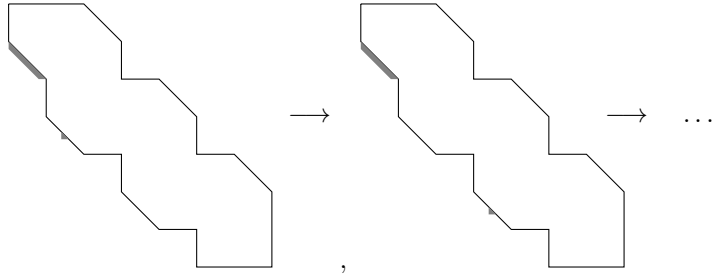
\[
    \begin{tikzpicture}[scale=.5]
     \draw (0,0) -- (2,0) -- (3,-1) -- (3,-2) -- (4,-2)
        -- (5,-3) -- (5,-4) -- (6,-4) -- (7,-5) -- (7,-7)
        -- (5,-7) -- (5,-6) -- (4,-6) -- (3,-5) -- (3,-4) -- (2,-4)
        -- (1,-3) -- (1,-2) -- (1,-2) -- (0,-1) -- (0,0);
        \fill[gray] (0,-1.2) -- (.8,-2.0) -- (1,-2) -- (0,-1);
        \fill[gray] (1.4,-3.4) -- (1.6,-3.6) -- (1.4,-3.6) -- (1.4,-3.4);
        \node at (8,-3) {$\longrightarrow$};
      \end{tikzpicture},~\begin{tikzpicture}[scale=.5]
     \draw (0,0) -- (2,0) -- (3,-1) -- (3,-2) -- (4,-2)
        -- (5,-3) -- (5,-4) -- (6,-4) -- (7,-5) -- (7,-7)
        -- (5,-7) -- (5,-6) -- (4,-6) -- (3,-5) -- (3,-4) -- (2,-4)
        -- (1,-3) -- (1,-2) -- (1,-2) -- (0,-1) -- (0,0);
        \fill[gray] (0,-1.2) -- (.8,-2.0) -- (1,-2) -- (0,-1);
        \fill[gray] (3.4,-5.4) -- (3.6,-5.6) -- (3.4,-5.6) -- (3.4,-5.4);
        \node at (8,-3) {$\longrightarrow$ ~~ \ldots };
      \end{tikzpicture} 
\]

We now continue to apply the rotations from the left, and this will create
bulges in the lower triangular structure of the block $(1,3)$. Similarly,
these bulges have a rank one structure and
they 
can be chased until the bottom of the matrix, as shown in the next
picture:
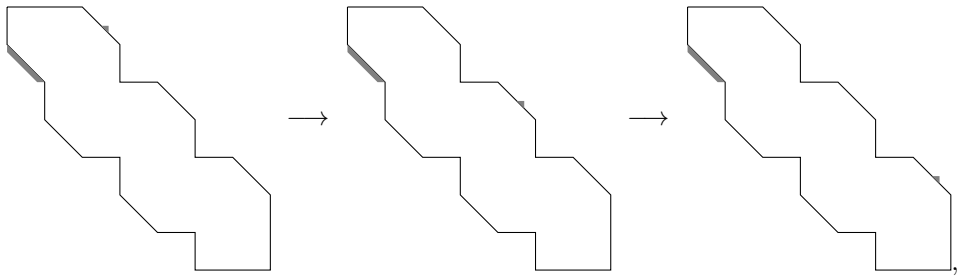
\[
  \begin{tikzpicture}[scale=.5]
     \draw (0,0) -- (2,0) -- (3,-1) -- (3,-2) -- (4,-2)
        -- (5,-3) -- (5,-4) -- (6,-4) -- (7,-5) -- (7,-7)
        -- (5,-7) -- (5,-6) -- (4,-6) -- (3,-5) -- (3,-4) -- (2,-4)
        -- (1,-3) -- (1,-2) -- (1,-2) -- (0,-1) -- (0,0);
        \fill[gray] (0,-1.2) -- (.8,-2.0) -- (1,-2) -- (0,-1);
        \fill[gray] (2.5,-.5) -- (2.7,-.7) -- (2.7,-.5) -- (2.5,-.5);
        \node at (8,-3) {$\longrightarrow$};
      \end{tikzpicture}~\begin{tikzpicture}[scale=.5]
     \draw (0,0) -- (2,0) -- (3,-1) -- (3,-2) -- (4,-2)
        -- (5,-3) -- (5,-4) -- (6,-4) -- (7,-5) -- (7,-7)
        -- (5,-7) -- (5,-6) -- (4,-6) -- (3,-5) -- (3,-4) -- (2,-4)
        -- (1,-3) -- (1,-2) -- (1,-2) -- (0,-1) -- (0,0);
        \fill[gray] (0,-1.2) -- (.8,-2.0) -- (1,-2) -- (0,-1);
        \fill[gray] (4.5,-2.5) -- (4.7,-2.7) -- (4.7,-2.5) -- (4.5,-2.5);
        \node at (8,-3) {$\longrightarrow$};
      \end{tikzpicture}~\begin{tikzpicture}[scale=.5]
     \draw (0,0) -- (2,0) -- (3,-1) -- (3,-2) -- (4,-2)
        -- (5,-3) -- (5,-4) -- (6,-4) -- (7,-5) -- (7,-7)
        -- (5,-7) -- (5,-6) -- (4,-6) -- (3,-5) -- (3,-4) -- (2,-4)
        -- (1,-3) -- (1,-2) -- (1,-2) -- (0,-1) -- (0,0);
        \fill[gray] (0,-1.2) -- (.8,-2.0) -- (1,-2) -- (0,-1);
        \fill[gray] (6.5,-4.5) -- (6.7,-4.7) -- (6.7,-4.5) -- (6.5,-4.5);
      \end{tikzpicture}, 
\]
and we continue the procedure until we absorb the bulge at the bottom
of the matrix. Handling the remaining rotations from the right is not
so easy. Applying them at this stage would alter
the first column of the block in position $(3,2)$ and create
a rank $1$ block in position $(3,1)$. Since we want to avoid this
fill-in we roll-up the first column of the block $(3,2)$ by using
some rotations on the left. As a consequence of the Nullity theorem,
this will automatically annihilate the diagonal entries of the lower
triangular matrix in position $(3,5)$ and, at the last step, the
first row of the block in position $(3,4)$. We perform a similar
transformation (operating only from the left) on the other lower blocks. 
We obtain the following
structure:
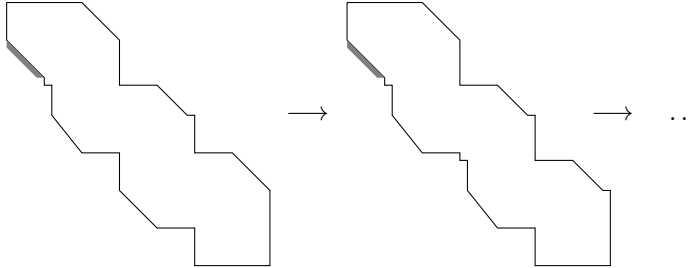
\[
  \begin{tikzpicture}[scale=.5]
     \draw (0,0) -- (2,0) -- (3,-1) -- (3,-2.2) -- (4,-2.2) 
        -- (4.8,-3) -- (5,-3) -- (5,-4) -- (6,-4) -- (7,-5) -- (7,-7)
        -- (5,-7) -- (5,-6) -- (4,-6) -- (3,-5) -- (3,-4) -- (2,-4)
        -- (1.2,-3) -- (1.2,-2.2) -- (1,-2.2) -- (1,-2) -- (1,-2) -- (0,-1) -- (0,0);
        \fill[gray] (0,-1.2) -- (.8,-2.0) -- (1,-2) -- (0,-1);
        \node at (8,-3) {$\longrightarrow$};
      \end{tikzpicture}~\begin{tikzpicture}[scale=.5]
     \draw (0,0) -- (2,0) -- (3,-1) -- (3,-2.2) -- (4,-2.2) 
        -- (4.8,-3) -- (5,-3) -- (5,-4.2) -- (6,-4.2) -- (6.8,-5) -- (7,-5) -- (7,-7)
        -- (5,-7) -- (5,-6) -- (4,-6) -- (3.2,-5) -- (3.2,-4.2) -- (3,-4.2) -- (3,-4) -- (2,-4)
        -- (1.2,-3) -- (1.2,-2.2) -- (1,-2.2) -- (1,-2) -- (1,-2) -- (0,-1) -- (0,0);
        \fill[gray] (0,-1.2) -- (.8,-2.0) -- (1,-2) -- (0,-1);
        \node at (8,-3) {$\longrightarrow$ ~~ \ldots };
      \end{tikzpicture}
\]
After this reduction we still have to perform some transformations
from the right. These transformations only operate on
odd-indexed blocks, and thus commute with transformations 
operating on even-indexed blocks. This will be important in the
following, since we keep these transformation ``in standby'', and we
will only apply them later.

We can now finally apply the transformations on the first block
column. The first rotation will propagate the only non-zero element
that we have left in the first column of the block $(3,2)$, and that
will extend the upper Hessenberg structure of the block in position
$(2,1)$. The following rotations will create bulges in the Hessenberg
structure, which we will chase to the bottom of the matrix.  Notice
that we are able to do so since all the rotations operate on
even-indexed blocks, thus commuting with the ones we have left ``in
standby''.

After the chasing procedure, we apply these last rotations. This creates some
fill-in, as reported in the following picture.
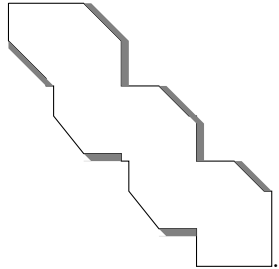
\[
  \begin{tikzpicture}[scale=.5]
     \draw (0,0) -- (2,0) -- (3,-1) -- (3,-2.2) -- (4,-2.2) 
        -- (4.8,-3) -- (5,-3) -- (5,-4.2) -- (6,-4.2) -- (6.8,-5) -- (7,-5) -- (7,-7)
        -- (5,-7) -- (5,-6) -- (4,-6) -- (3.2,-5) -- (3.2,-4.2) -- (3,-4.2) -- (3,-4) -- (2,-4)
        -- (1.2,-3) -- (1.2,-2.2) -- (1,-2.2) -- (1,-2) -- (1,-2) -- (0,-1) -- (0,0);
        \fill[gray] (0,-1.2) -- (1,-2.2) -- (1.2,-2.2) -- (0,-1);
        \fill[gray] (2,-4) -- (3,-4) -- (3,-4.2) -- (2,-4.2) -- (2.2,-4.2) -- (2,-4);
        \fill[gray] (4,-6) -- (5,-6) -- (5,-6.2) -- (4,-6.2) -- (4.2,-6.2) -- (4,-6);
        \fill[gray] (2,0) -- (2.2,0) -- (3.2,-1) -- (3.2,-2.2) -- (3,-2.2) -- (3,-1) -- (2,0);
        \fill[gray] (4,-2.2) -- (4.2,-2.2) -- (5.2,-3.2) -- (5.2,-4.2) -- (5,-4.2) -- (5,-3.2) -- (4,-2.2);
        \fill[gray] (6,-4.2) -- (6.2,-4.2) -- (7,-5) -- (6.8,-5) -- (6,-4.2);
      \end{tikzpicture}. 
\]
We notice that the fill-in on the lower triangular matrices in the
second block superdiagonal has re-introduced the elements that were
annihilated due to the Nullity theorem, but shifted down and right
by one. In fact, a careful look at the matrix reveals that the complete
CMV structure has been shifted down and right by one entry, and thus
removing the first row and column give us a matrix with the same
structure, as desired.
We summarize the previous procedure in the following result. 

\begin{theorem}
  \label{thm:cmv-chasing}
  Let $A_{CMV}$ a matrix in block CMV form, as defined previously.
  Let $P$ a sequence of $2k - 2$ Givens rotations acting on
  the rows $2, \ldots, 2k$, starting from the bottom to the top.
  Then there exist a unitary matrix
  $Q$ such that
  \[
    \tilde A = Q P A_{CMV} P^H Q^H
  \]
  has the same block CMV structure of $A_{CMV}$ in the trailing
  matrix obtained removing the first row and the first column.
  Moreover, $Q P U$ has only the first $k + 1$ rows
  different from zero. 
\end{theorem}

\begin{proof}
  Follow the steps described in the previous pages. The claim
  on the structure of $U$ after the transformation follows by
  noticing the all the transformations except the ones
  contained in $P$ operates on rows with indices larger than
  $k + 1$. 
\end{proof}

The same procedure can be applied in all the steps of Hessenberg reduction,
and the CMV structure shifts down and right of one entry at each step.
This provides a concrete procedure that allows to perform
the reduction to upper Hessenberg form in a cheap way.

We can perform a flops count for the above algorithm, to ensure
that the complexity is the one expected. First, notice that applying
a Givens rotation to the matrix $A$ has a cost of $\mathcal O(k)$ flops,
given the banded structure that we are preserving. Following the algorithm
in Theorem~\ref{thm:cmv-chasing} shows that we only need $\mathcal O(n)$
rotations to perform a single step, therefore we have 
$\mathcal O(nk)$ flops per step, which yields a total cost of
$\mathcal O(n^2k)$ operations. 


However, 
whilst the algorithm allows to compute the Givens rotations needed
at each step in order to continue the process, it is not clear,
at this stage, how to recover the reduced matrix at the end of the
procedure.
This is the topic of the next section.

\begin{remark}
  Notice that, after the first step, the transformed matrix $\tilde A$ has
  the trailing part in block CMV form, but the trailing matrix is not unitary.
  Therefore, in order to apply the algorithm in the following steps (and still
  make use of the Nullity theorem), it is necessary to formally consider the
  entire matrix, and just limit the chasing on the trailing part. Also  notice that
  in the proof of Theorem~\ref{thm:cmv-chasing} we rely on the Nullity  theorem  to
  guarantee that some elements have to be zero.  This does not happen in practice due
  to floating-point arithmetic but 
  since everything here is operating on a unitary matrix
  after some steps we have $\mathrm{fl}( QA_{CMV}Q^H )=QA_{CMV}Q^H +\Delta$ 
  and the entries in the perturbation
  have to be $\approx u$ norm-wise without error  amplification. 
\end{remark}

\subsection{Representing the final matrix}\label{33}

The previous section provides  a concrete way to compute the
transformations required to take the diagonal matrix
plus low-rank correction into upper-Hessenberg form.
The computation of the transformations is possible since,
at each step, we have a structured representations of the
part of the matrix that needs to be reduced. However,
with the previous approach, we do not have an explicit
expression of the part of the matrix that has already
been reduced to upper Hessenberg form. 
The aim of this section is to show
how to represent that part.

To perform this analysis, we use the following
notation: We  denote by $P_j$ and by  $Q_j$, respectively,  the product of $2k-2$ rotations
  that cleans the elements in the $j$-th column in order to make 
  the matrix upper Hessenberg in that column and 
 the product of rotations applied to perform the subsequent  bulge chasing. Thus the matrix after  $j$ steps  is
  $Q_j P_j \ldots Q_1 P_1 (A_{CMV} + UV^H) P_1^H Q_1^H\ldots P_j^H Q_j^H$  and it will continue to  
  have a $2k$-banded structure 
  in the bottom-right part that still have to be reduced. 
  
  Let then $PA_{CMV}P^H$  be the the unitary matrix obtained at the end
  of the process, that is $P = Q_{n-2} P_{n-2} \ldots Q_1 P_1$. 
  
  \begin{lemma}
    Each row of $PA_{CMV}P^H$ admits a representation of the form
    \[
      e_j^H PA_{CMV}P^H = 
        [\B  w^H \ 0 \ \ldots \ 0 ] P_{j+1}^H \ldots P_{n-2}^H, 
    \]
    where $\B w \in \mathbb{C}^{j + 2k}$ for any $j \leq n - 2k$. 
  \end{lemma}
  
  \begin{proof}
    It follows directly by the preservation of the CMV structure
    (and so in particular of the $2k$-banded structure) and 
    by noting that after reducing the $j$-th column the 
    rotations on the left will not alter the $j$-th row
    anymore, and the ones on the right will update it 
    only when performing the first cleaning step, but 
    not the bulge chasing (since these will act on zero components). 
  \end{proof}
  
  Notice that a similar result also holds for the
  columns, since the situation is completely symmetric. 
  The above representation is usually known as 
  Givens--Vector representation \cite{vanbarel:book1,vanbarel:book2}, 
  and can also be seen as a particular 
  case of generators-based representation \cite{eidelman:book1,eidelman:book2}. 
  Being the matrix upper Hessenberg we only need to store the upper triangular
  part plus the subdiagonal, so the above representation needs only
  $O(nk)$ storage, since we need to represent $O(n)$ vectors of length
  $k$ and $O(nk)$ rotations.

  \section{Numerical results}\label{sec:numerical}

  In this section we report numerical experiments that validate
  the proposed algorithm. We tested the accuracy and the runtime
  of our implementations.

  The software can be downloaded for further testing
  at
  \url{http://numpi.dm.unipi.it/software/rank-structured-hessenberg-reduction}.
  The package contains two functions: 
  \begin{description}
  \item[\texttt{rshr\_dlr}] implements the Hessenberg reduction for a real diagonal
    plus low rank matrix. This function is implemented in MATLAB and FORTRAN, and
    a MATLAB wrapper to the FORTRAN code is provided and used automatically if
    possible. 
  \item[\texttt{rshs\_ulr}] implements the analogous reduction for the unitary
    plus low-rank case. It is only implemented as a MATLAB function. 
  \end{description}

  \subsection{The real diagonal plus low rank case}

  We have
  checked the complexity of our reduction strategy by fixing $n$
  and $k$, respectively, and comparing the runtime while
  varying the other parameter. We verified that 
  the complexity is quadratic in $n$ and 
  linear in $k$. 

  When measuring the complexity in $n$ we have also compared
  our timings with the \texttt{DGEHRD} function included in
  LAPACK 3.6.1.

  The tests have been run on a server with an Intel Xeon CPU
  E5-2697 running at 2.60GHz, and with 128 GB of RAM. The amount
  of RAM allowed us to test \texttt{DGEHRD} for large dimensions, which
  would have not been feasible on a desktop computer.

  The results that we report are for compiler optimizations turned off.
  This makes the dependency of the timings on the size
  and rank much clearer in the plots,
  and allows us to clearly connect the timings with the flop count. However,
  each experiment has been repeated with optimization turned on that 
  yields a gain of a factor between $2$ and $4$ in
  the timings, both for \texttt{DGEHRD} and for our method,
  so there is no significant difference when comparing the
  approaches. 

  \begin{figure}[t]
    \centering
    \begin{tikzpicture}
      \begin{loglogaxis}[
        legend pos = south east,
        width=\linewidth,
        height=.55\textwidth,
        xlabel = $n$,
        ylabel = Time (s)
        ]
        \addplot table {hermitian_tn.dat};
        \addplot table[y index=2] {hermitian_tn.dat};                
        \addplot table {hermitian_tnf.dat};
        \addplot[domain=8:16384, gray, dashed] {7e-8 * x^2};
        \legend{\texttt{rshr\_dlr} ($k = 4$), \texttt{rshr\_dlr} ($k = 32$),
          \texttt{DGEHRD} in \texttt{LAPACK}, $\mathcal O(n^2)$};
      \end{loglogaxis}
    \end{tikzpicture}
    \caption{Complexity of the Hessenberg reduction method in
      dependency of the size $n$. The quasiseparable rank $k$ has
      been chosen equal to $4$ in this example. In each example
      the tests have been run only when $k < n$.}
    \label{fig:hermitian_comp_n}
  \end{figure}
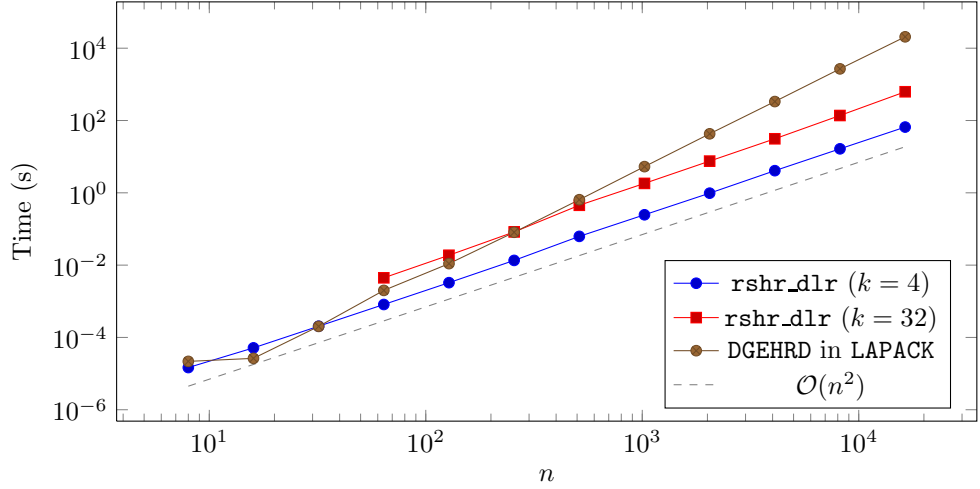

  Figure~\ref{fig:hermitian_comp_n} shows the complexity of our
  method as a function of the size $n$. The behavior is clearly
  quadratic, as shown by the dashed line in the plot. It is also
  clear that the method is faster than LAPACK for relatively
  small dimensions (in this case, $n \approx 16$) if the rank is
  small. Given that the ratio between the two complexities is
  $n^3 / n^2 k = \frac{n}{k}$ this suggests that our approach is
  faster when $k \lessapprox \frac{n}{8}$. As an example, consider the timings
  for $k = 32$ that are also reported in Figure~\ref{fig:hermitian_comp_n};
  we expect our approach to be faster when $n > 8k = 256$, and
  this is confirmed by the plot. 

  In particular, our constant in front of the complexity is just
  $8$ times larger than a dense Hessenberg reduction. We believe
  that a clever use of blocking might even improve this result,
  and could be the subject of future investigation. 
  
    \begin{figure}[t]
    \centering
    \begin{tikzpicture}
      \begin{loglogaxis}[
        legend pos = south east,
        width=\linewidth,
        height=.55\textwidth,
        xlabel = $k$,
        ylabel = Time (s)
        ]
        \addplot table {hermitian_tk.dat};
        \addplot table[x index =0, y index = 2] {hermitian_tk.dat};
        \addplot[domain=1:256, samples=50, gray, dashed] {1.5e-1 * x};
        \legend{\texttt{rshr\_dlr}, \texttt{DGEHRD} in \texttt{LAPACK}, $\mathcal O(n)$};
      \end{loglogaxis}
    \end{tikzpicture}
    \caption{Complexity of the Hessenberg reduction method in
      dependency of the quasiseparable rank $k$.
      The dimension $n$ has
      been chosen equal to $2048$ in this example, and
      rank varies from $1$ to $256$.}
    \label{fig:hermitian_comp_k}
  \end{figure}
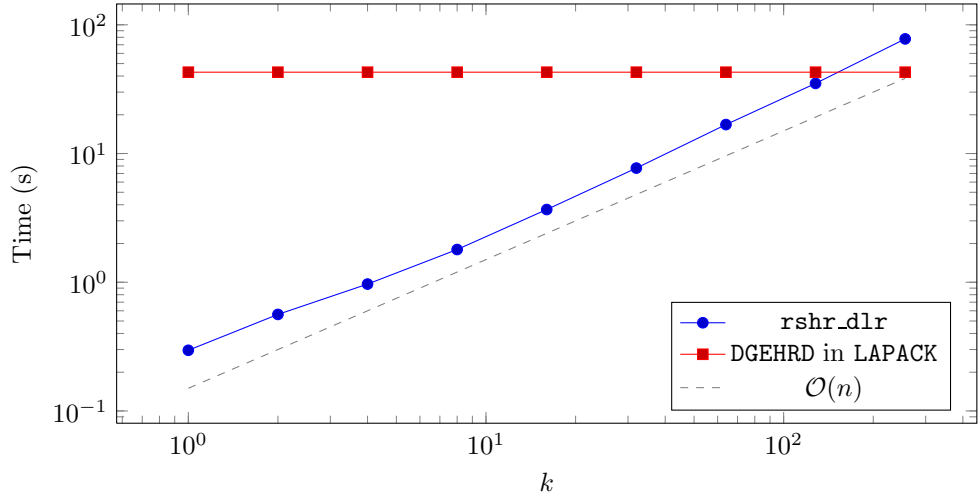

  We performed a similar analysis for the cost as a function
  of the rank. Figure~\ref{fig:hermitian_comp_k} shows that
  the growth of the complexity with respect to $k$ is essentially
  linear. We have taken a matrix of size $2048$ and we have compared
  the performance of our own implementation with the one in
  LAPACK. It is worth noting that since our implementation
  only requires $\mathcal O(nk)$ storage if $n$ grows large enough
  then it also provides an alternative when a dense reduction
  cannot be applied due to memory constraints. Moreover, the linear
  behavior is very clear from the plot. 

  Our implementation relies on a compressed diagonal
  storage scheme (CDS) in order to efficiently operate
  on the small bandwidth, and it never needs to store
  the large $n \times n$ matrix.

  Finally, we have tested the backward stability of our approach.  In
  order to do so we have computed the Hessenberg form with our fast
  method, along with the change of basis $Q$ by accumulating all
  the plane rotations. Then, we have measured
  the norm of $A - Q^HHQ$. The results are reported in
  Figure~\ref{fig:hermitian_back}.

    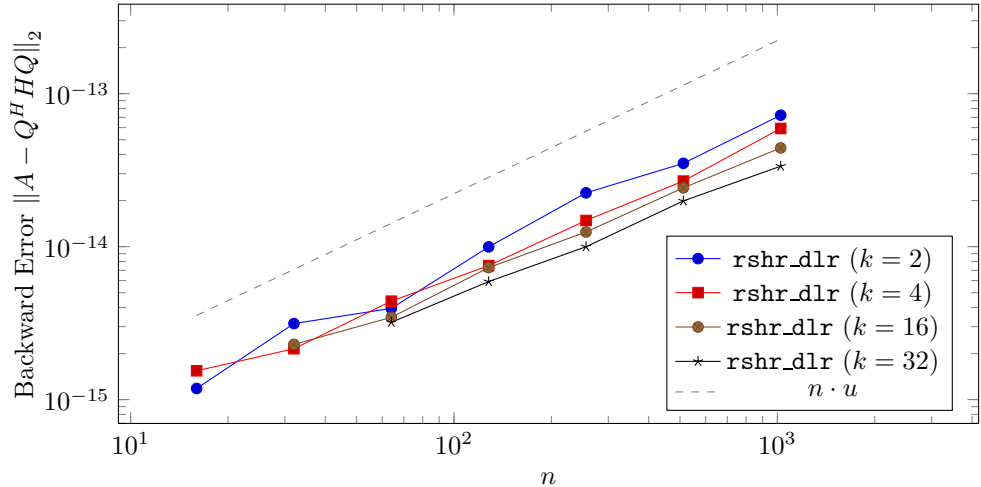
\begin{figure}[t]
    \centering
    \begin{tikzpicture}
      \begin{loglogaxis}[
        legend pos = south east,
        width=\linewidth,
        height=.55\textwidth,
        xlabel = $n$,
        ylabel = Backward Error $\norm{A - Q^HHQ}_2$,
        xmax = 4192
        ]
        \addplot table {hermitian_backward.dat};
        \addplot table[y index = 2] {hermitian_backward.dat};
        \addplot table[y index = 3] {hermitian_backward.dat};
        \addplot table[y index = 4] {hermitian_backward.dat};
        \addplot[domain=16:1024, dashed, gray] {2.22e-16 * x};
        \legend{\texttt{rshr\_dlr} ($k = 2$),
          \texttt{rshr\_dlr} ($k = 4$),
          \texttt{rshr\_dlr} ($k = 16$),
          \texttt{rshr\_dlr} ($k = 32$),
          $n \cdot u$
        }
      \end{loglogaxis}
    \end{tikzpicture}
    \caption{Relative backward error on the original matrix of the computed
      Hessenberg form, for various values of $n$ and $k$ in
      the real diagonal plus low rank case.
      The dashed line represents the machine precision
      multiplied by the size of the matrix.}
    \label{fig:hermitian_back}
  \end{figure}

  The experiments show that our method is backward stable with
  a backward error smaller than $\mathcal O(nu)$, where $n$
  is the size of the matrix and $u$ the unit roundoff. In practice
  one can prove from Algorithm in Section~\ref{recomul} that
  the backward stability is achieved with a low-degree polynomial
  in $n$ and $k$ times the unit roundoff, obtaining an upper
  bound of the form $\mathcal O(n^2 k u)$.

  The numerical results have been obtained by averaging different
  runs of the algorithm on random matrices.
  
\subsection{The unitary case}

In the unitary diagonal plus low rank case we have used our MATLAB
implementation to test the accuracy of the algorithm. We also report
some results on the timings that show that the asymptotic complexity
is the one we expect.

The implementation of the unitary plus low-rank case is slightly
more involved of the Hermitian one, mainly due to the bookkeeping
required to handle the application of plane rotations in the
correct order. Nevertheless, the software developed for the
real diagonal plus low rank case could be adapted to efficiently
handle also this case. 

    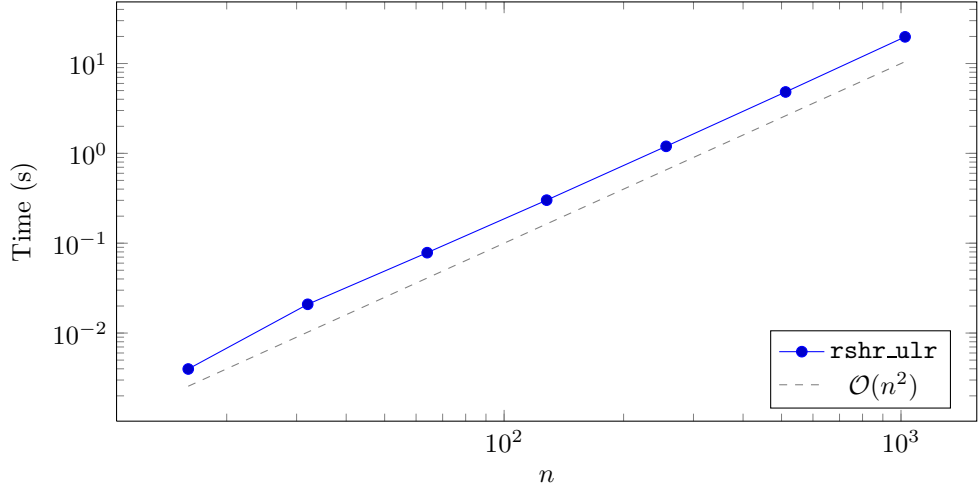
\begin{figure}[t]
    \centering
    \begin{tikzpicture}
      \begin{loglogaxis}[
        legend pos = south east,
        width=\linewidth,
        height=.55\textwidth,
        xlabel = $n$,
        ylabel = Time (s)
        ]
        \addplot table {unitary_tn.dat};
        \addplot[domain=16:1024, samples=50, gray, dashed] {1e-5 * x^2};
        \legend{\texttt{rshr\_ulr}, $\mathcal O(n^2)$};
      \end{loglogaxis}
    \end{tikzpicture}
    \caption{Complexity of the Hessenberg reduction method in
      dependency of the dimension $n$
      the unitary plus low rank case. 
      The dimension $n$ has been chosen between $16$ and $1024$, with
      a fixed rank $k = 4$.}
    \label{fig:unitary_comp_n}
  \end{figure}

      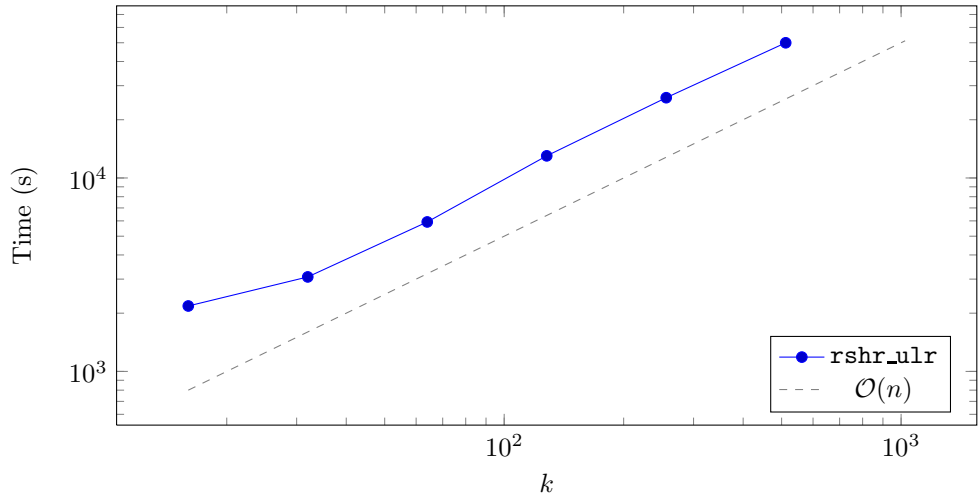
\begin{figure}[t]
    \centering
    \begin{tikzpicture}
      \begin{loglogaxis}[
        legend pos = south east,
        width=\linewidth,
        height=.55\textwidth,
        xlabel = $k$,
        ylabel = Time (s)
        ]
        \addplot table {unitary_tk.dat};
        \addplot[domain=16:1024, samples=50, gray, dashed] {.5e2 * x};
        \legend{\texttt{rshr\_ulr}, $\mathcal O(n)$};
      \end{loglogaxis}
    \end{tikzpicture}
    \caption{Complexity of the Hessenberg reduction method in
      dependency of the quasiseparable rank $k$ for
      the unitary plus low rank case. 
      The rank $k$ has been chosen between $16$ and $512$, with
      a fixed size $n = 8192$.}
    \label{fig:unitary_comp_k}
  \end{figure}

  We have reported the results for the timings (taken in MATLAB)
  for our implementation of the Hessenberg reduction of unitary
  plus low rank matrices in Figure~\ref{fig:unitary_comp_n} and
  Figure~\ref{fig:unitary_comp_k}. These two figures report the
  dependency on the size (which is quadratic) and on the quasiseparable
  rank $k$ (which is linear).

  We believe that an efficient implementation of this reduction
  could be faster than LAPACK for small dimension, as it happens
  for the real diagonal plus low rank case. However, due to the
  increased complexity, it is likely that this method will be
  faster for slightly larger ratios of $\frac n k$, compared to
  what we have for the real diagonal plus low rank case. 

    \begin{figure}[t]
    \centering
    \begin{tikzpicture}
      \begin{loglogaxis}[
        legend pos = south east,
        width=\linewidth,
        height=.55\textwidth,
        xlabel = $n$,
        ylabel = Backward Error $\norm{A - Q^H HQ}_2$,
        xmax = 4192
        ]
        \addplot table {unitary_backward.dat};
        \addplot table[y index = 2] {unitary_backward.dat};
        \addplot table[y index = 3] {unitary_backward.dat};
        \addplot table[y index = 4] {unitary_backward.dat};
        \addplot[domain=16:1024, dashed, gray] {2.22e-16 * x};
        \legend{\texttt{rshr\_ulr} ($k = 2$),
          \texttt{rshr\_ulr} ($k = 4$),
          \texttt{rshr\_ulr} ($k = 16$),
          \texttt{rshr\_ulr} ($k = 32$),
          $n \cdot u$
        }
      \end{loglogaxis}
    \end{tikzpicture}
    \caption{Relative backward error on the original matrix of the computed
      Hessenberg form, for various values of $n$ and $k$ in
      the unitary plus low rank case.
      The dashed line represents the machine precision
      multiplied by the size of the matrix.}
    \label{fig:unitary_back}
  \end{figure}
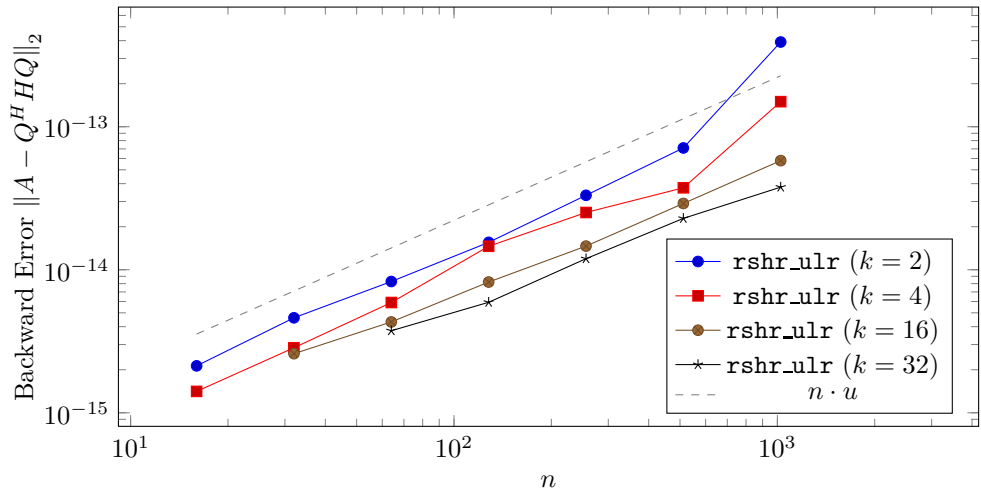

  Similarly to the Hermitian case, we have also performed numerical
  experiments with the purpose of estimating the backward error.
  The results are reported in Figure~\ref{fig:unitary_back}
  and are completely analogous
  to the ones obtained for the real diagonal plus low rank case. 

\section{Conclusions and Future Work}\label{sec:conclusions} 
In this paper we have presented a fast algorithm for reducing a
$n\times n$ real/unitary diagonal $D$ plus a low rank correction in
upper Hessenberg form. To our knowledge this is the first algorithm
which is efficient w.r.t the size $k$ of the correction by requiring
$O(n^2 k)$ flops and provides a viable alternative to
the dense LAPACK routine for small sizes.
The approach for the unitary case relies upon some
theoretical and computational 
properties of block unitary CMV matrices  of   independent
interest  for computations with orthogonal  matrix polynomials.
The application of our algorithm for solving
interpolation-based linearizations of generalized eigenvalue problems
 is still an ongoing research.

 In particular, the output of both reduction algorithms is already
 in a rank structured form which could be exploited in an
 iteration (such as the QR method) to compute the eigenvalues
 of the matrix $A$. This could provide a method with optimal
 complexity for the computation of eigenvalues of matrix polynomials,
 expressed in monomial and interpolation bases.

\bibliographystyle{siamplain}
\bibliography{redbib}
\end{document}